\documentclass[12pt,reqno]{amsart}
\usepackage{amsmath,amssymb,amsfonts,amsthm}
\usepackage[right=2.5cm, left=2.5cm, top=2.5cm, bottom=2.5cm]{geometry}
\usepackage{tikz}
\usetikzlibrary{positioning, arrows.meta, shapes.geometric}
\usetikzlibrary{matrix}
\usepackage{graphicx}
\usepackage{pgfplots}
\usepackage{subcaption}
\usepgfplotslibrary{fillbetween}
\usetikzlibrary{patterns}
\DeclareMathAlphabet{\mathpzc}{OT1}{pzc}{m}{it}

\newtheorem{theorem}{Theorem}[section]
\newtheorem{lema}[theorem]{Lemma}

\theoremstyle{definition}
\newtheorem{de}[theorem]{Definition}

\theoremstyle{remark}
\newtheorem{remark}[theorem]{Remark}

\numberwithin{equation}{section}

\begin{document}
	
\title[Navier-Stokes equations with dual-scale hereditary viscosity]{The Navier-Stokes equations with dual-scale hereditary viscosity: supercritical norm inflation and global well-posedness in critical spaces}
	
	\author[B. de Andrade]{Bruno de Andrade}
	\address[B. de Andrade]{Departamento de Matem\'atica\\
		Universidade Federal de Sergipe\\
		S\~ao Crist\'ov\~ao - SE\\
		Brazil.}
	\email{bruno@mat.ufs.br}
	
\begin{abstract}
	This manuscript investigates the Cauchy problem for an incompressible fluid flow governed by a dual-scale hereditary memory, representing a viscoelastic variant of the classical Navier-Stokes equations that captures anomalous momentum transport. The non-local dissipation breaks exact global scale invariance, dictating a pseudo-differential analysis within the H\"{o}rmander class $S^{-2}_{1,0}$ where the spatial gradient induces a fractional temporal penalty. We rigorously establish $L^q-L^p$ decay estimates and identify the critical Lebesgue threshold $p_c = N (\frac{1+\alpha_\infty}{1-\alpha_\infty})$. In the supercritical regime $1 < p < p_c$, we bypass the loss of spatial localization by mapping the frequency-modulated bilinear flow directly into Fourier space; by utilizing Bernstein's inequalities, we prove instantaneous norm inflation  at the origin and confirm intrinsic ill-posedness. Conversely, in the topological limit $p \to \infty$, we demonstrate that the dual-scale memory structurally prevents the collapse traditionally observed for classical fluids within the maximal critical Besov space $\dot{B}^{-1}_{\infty, \infty}$. By exploiting an asymmetric interpolation within Bony's para-differential calculus, we prove that the temporal smoothing  overpowers the high-high convective resonant cascade. This delicate analytical balance confines ill-posedness to the non-separable high-frequency tail of the Besov topology, thereby establishing global-in-time Hadamard well-posedness for small initial data possessing high-frequency adherence within $\dot{B}^{-\kappa}_{\infty, \infty}(\mathbb{R}^N)$, a well-posedness regime strictly broader than the separable little Besov closure $\dot{b}^{-\kappa}_{\infty, \infty}(\mathbb{R}^N)$, where $\kappa = \frac{1-\alpha_\infty}{1+\alpha_\infty}$.
\end{abstract}


\maketitle

\section{Introduction}

The mathematical theory of incompressible fluid dynamics is classically anchored by the Navier-Stokes equations, whose well-posedness in critical function spaces remains one of the most prominent challenges in nonlinear analysis. While the classical model is a triumph of continuum mechanics, the increasing need to understand complex fluid behaviors has driven the community to investigate non-local and anomalous momentum transport. In this manuscript, we establish the exact topological boundaries for the existence, stability, and supercritical collapse of incompressible flows governed by a non-local, dual-scale hereditary memory.

\subsection{Physical motivation and the multi-scale memory kernel}
To capture the dynamics of complex media, we replace the instantaneous kinematic viscosity of the classical model with a hereditary integration operator. The resulting initial-value problem, which we denote as the Navier-Stokes equations with Hereditary Viscosity (NSHV), governs the velocity field $u(x,t)$ and the hydrostatic pressure $p(x,t)$ as follows
\begin{equation*}
	\begin{cases}
		u_t = \displaystyle\int_0^t g(t-s) \Delta u(x,s) \,ds - (u \cdot \nabla) u - \nabla p, \quad \text{in } \mathbb{R}^N \times (0, \infty),\\
		\nabla \cdot u = 0, \quad \text{in } \mathbb{R}^N \times (0, \infty), \\
		u(x,0) = u_0(x), \quad \text{in } \mathbb{R}^N.
	\end{cases}
\end{equation*}

Projecting the system onto the space of distributionally divergence-free vector fields via the Leray-Hopf operator $\mathbb{P}$ canonically rewrites the convective term into the abstract form
\begin{equation}\label{eq:NSHV_projected}
	u_t = \int_0^t g(t-s) \Delta u(x,s) \,ds - \mathbb{P}\nabla \cdot (u \otimes u).
\end{equation}
The introduction of the memory kernel $g \in L^1_{\text{loc}}(\mathbb{R}^+)$ is not merely a mathematical abstraction, but a phenomenological necessity in the rational mechanics of non-Newtonian fluids and viscoelastic turbulence (see Barbu and Sritharan \cite{barbu2003}). In highly heterogeneous fluids---such as polymer melts, micellar networks, and concentrated suspensions---momentum transport exhibits a fading memory effect.

In this work, we assume that $g$ is a \textit{dual-scale}  admissible kernel. Much of the classical literature on fractional viscoelasticity employs single-scale power-law kernels (equivalent to standard Caputo or Riemann-Liouville fractional derivatives). However, as extensively documented in modern rheology (e.g., Jaishankar and McKinley \cite{jaishankar2013} and Mainardi \cite{mainardi2010}), a single power-law implies an infinite zero-shear viscosity and fails to capture the terminal relaxation of some  real complex fluids. The dual-scale kernel resolves this physical paradox: its Laplace transform encodes an anomalous, highly elastic response at short times (high frequencies) governed by an exponent $\alpha_\infty \in (0,1)$, which continuously crosses over to a distinct fluid-like terminal relaxation at long times (low frequencies) characterized by $\alpha_0 \in [0,1)$.

From a mathematical perspective, this physical fidelity comes at a severe analytical cost: the continuous crossover between distinct asymptotic regimes permanently breaks the exact global scale invariance of the system.

\subsection{State of the art and mathematical challenges}

A central feature of the classical Navier-Stokes model is its exact scale invariance, which dictates the topological boundaries for local and global existence. The quest for well-posedness in scale-invariant functional spaces has driven decades of profound harmonic analysis. This trajectory was rigorously inaugurated by Kato \cite{kato1984}, who established the well-posedness of strong solutions in the critical Lebesgue space $L^N(\mathbb{R}^N)$. Subsequent advancements led to seminal results in critical Besov spaces by Cannone \cite{cannone1995}, culminating in the optimal well-posedness threshold in $BMO^{-1}$ by Koch and Tataru \cite{koch2001}. However, this classical framework suffers from a severe topological collapse in the maximal scale-critical Besov space $\dot{B}^{-1}_{\infty, \infty}(\mathbb{R}^N)$. As established by Bourgain and Pavlovi\'c \cite{bourgain2008}, and further geometrically refined by Cheskidov and Shvydkoy \cite{cheskidov2010}, the classical flow is intrinsically ill-posed in this limit. The mechanism of failure relies on a catastrophic low-frequency resonance: high-frequency convective interactions generate a continuous energy cascade towards the macroscopic modes, causing instantaneous norm inflation even for arbitrarily small initial data.

Determining whether fractional and integro-differential variants of fluid models inherit this catastrophic collapse requires a robust mild framework. The rigorous mathematical analysis of the NSHV equations was inaugurated in \cite{barbu2003}, establishing local solvability in the $L^2$-setting via an $m$-accretive quantization. Integrating fractional calculus with abstract evolution theory, de Andrade, Silva, and Viana \cite{deandrade2021} advanced the model to $L^q$-settings for power-type materials. By utilizing the concept of $\epsilon$-regular mild solutions \cite{arrieta2000}, they circumvented the temporal singularity at the origin inherent to singular Volterra convolutions. The robustness of this abstract integro-differential framework was broadly generalized by de Andrade and Viana \cite{deandrade2017} to handle generic parabolic models with memory driven by critical nonlinearities. Pushing the topological boundaries further, de Andrade, Cuevas, and Dantas \cite{deandrade2024} established well-posedness in the expansive framework of homogeneous Besov-Morrey spaces, while de Andrade and Santana \cite{deandrade_santana} provided a unified abstract theory for local existence and blow-up alternatives in interpolation scales.

Despite these significant milestones, previous approaches to the NSHV equations invariably relied on single-scale fractional models or operated strictly within subcritical regimes. In those specific contexts, the single-scale kernels preserve a form of global scale invariance, allowing the linear resolvent to be treated via classical subordination principles and standard Wright functions. By introducing the dual-scale memory kernel, this analytical convenience breaks down entirely. The inherent absence of a global self-similar scaling dictates that the problem must be tackled via a highly refined pseudo-differential harmonic analysis.

The core mathematical challenge lies in resolving a delicate analytical balance. On one side, the Calder\'on-Zygmund nature of the Leray-Hopf projector coupled with the spatial gradient $\nabla$ forces the dispersion of momentum, continually attempting to drive the bilinear flow into instantaneous ill-posedness. On the other side, the dual-scale hereditary integration operator acts as a pseudo-differential smoothing mechanism, suppressing high-frequency oscillations via algebraic tails in Fourier space.

\subsection{Main contributions and methodology}
Our analysis relies on the pseudo-differential framework developed in \cite{deandrade2026} combined with Bony's para-differential calculus and Littlewood-Paley dyadic decompositions \cite{sawano2018}. The central results are summarized as follows:

\textbf{A. Linear estimates and topological barriers:} The exact physical symbol of the dual-scale resolvent globally belongs to the H\"ormander multiplier class $S^{-2}_{1,0}$. Consequently, the application of the spatial gradient exacts an anomalous temporal penalty. We derive sharp $L^q-L^p$ decay estimates, revealing that the local integrability of the associated singular kernel is governed by the structural constraint $1/q - 1/p < 1/N$. This geometric barrier identifies the critical Lebesgue threshold for local well-posedness as $p_c = N(\frac{1+\alpha_\infty}{1-\alpha_\infty})$. The parameter $\alpha_\infty \in (0,1)$ actively modulates this threshold: as $\alpha_\infty \to 0$, we recover the classical Navier-Stokes critical exponent $p_c \to N$; conversely, as $\alpha_\infty \to 1^-$, the temporal smoothing deteriorates, driving $p_c \to \infty$, thereby highlighting the severe topological penalty exacted by strong initial hereditary elasticity.

\textbf{B. Supercritical norm inflation ($1 < p < p_c$):} For Lebesgue topologies strictly below the critical threshold, the space lacks the structural capacity to absorb the convective loss of regularity. By adapting the frequency modulation technique of Christ, Colliander, and Tao \cite{Christ2003}, we demonstrate that the non-linear flow undergoes instantaneous norm inflation at the origin. A crucial harmonic innovation in our proof is the utilization of compactly supported spectra and Bernstein's inequalities, which elegantly circumvent the failure of the Hausdorff-Young inequality. This establishes that the singular integral flow is genuinely ill-posed without relying on localized spatial bounds, which are invariably destroyed by the Leray-Hopf projector.

\textbf{C. Well-posedness in the critical Besov limit:} As $p \to \infty$, the well-posedness boundary merges into the scale-critical Besov space $\dot{B}^{-\kappa}_{\infty, \infty}(\mathbb{R}^N)$, where $\kappa = \frac{1-\alpha_\infty}{1+\alpha_\infty}$. To determine whether the NSHV equations inherit the Bourgain-Pavlovi\'c catastrophe, we dissect the non-linear flow using Bony's paraproduct. By introducing asymmetric bilinear estimates within the high-high residue term, we interpolate the spatial base topology with a fractional temporal weight $t^\gamma$. This asymmetric interpolation avoids non-integrable temporal singularities at the origin and proves that the hereditary memory strictly overpowers the resonant cascade before it can inflate the macroscopic modes. Consequently, we secure global-in-time Hadamard well-posedness for small data in a class strictly broader than the separable little Besov closure $\dot{b}^{-\kappa}_{\infty, \infty}(\mathbb{R}^N)$, thereby mapping the exact topological boundary of the non-local parabolic regularization.

\subsection{Structure of the paper}
The manuscript is organized as follows. Section 2 formalizes the mathematical setting, detailing the structural hypotheses on the dual-scale memory kernel and the mild integral framework. Section 3 reviews the preliminary harmonic analysis tools, including Littlewood-Paley operators, Besov spaces, and the scale-invariant properties of the Leray-Hopf multiplier. In Section 4, we establish the $L^q-L^p$ singular decay estimates via pseudo-differential analysis. Section 5 identifies the critical Lebesgue threshold $p_c$, proving local well-posedness in the subcritical regime, followed by the rigorous proof of instantaneous norm inflation for $1 < p < p_c$. Finally, Section 6 establishes global well-posedness in the maximal critical limit via a contraction mapping.

	\section{Mathematical Setting and the Dual-Scale Resolvent}
	
\subsection{Admissible kernels}
The underlying linear dynamics are governed by the memory kernel $g \in L^1_{\text{loc}}(\mathbb{R}^+)$, whose structural behavior is characterized by the analytic properties of its Laplace transform $\hat{g}(\lambda) = \mathcal{L}\{g(t)\}(\lambda)$. This spectral representation, whose mathematical framework traces back to the classical theory of Widder \cite{Widder1941}, effectively resolves the high-frequency regime while systematically accommodating potential crossovers to terminal Newtonian relaxation or alternative viscoelastic behaviors at low frequencies. Throughout this work, we assume that $g$ belongs to the class of admissible kernels, satisfying the following structural hypotheses:
\begin{itemize}
	\item \textbf{(H1) Asymptotic sectoriality, regularity, and compatibility:} The function $\hat{g}(\lambda)$ remains analytic and non-vanishing in $\mathbb{C} \setminus (-\infty, 0]$. Furthermore, there exists a continuous function $\theta: (0, \infty) \to (0, \pi)$ satisfying $\inf_{r > 0} \theta(r) > 0$, which guarantees that the roots of the characteristic equation $\lambda + |\xi|^2 \hat{g}(\lambda) = 0$ strictly lie within the region $|\arg \lambda| \le \pi - \theta(|\lambda|)$. As $|\lambda| \to \infty$, the sector aperture stabilizes to $\theta(|\lambda|) \to \theta_\infty > 0$, which satisfies the high-frequency compatibility condition
	\begin{equation*}
		\theta_\infty < \frac{\pi \alpha_\infty}{1+\alpha_\infty},
	\end{equation*}
	for some $\alpha_\infty \in (0,1)$.
	\item \textbf{(H2) High-frequency asymptotics (short time):} There exist constants $C_1, C_2 > 0$ such that, for $|\lambda|$ sufficiently large within the principal sector $\Sigma_{\theta_\infty}$:
	\begin{equation*}
		C_1 |\lambda|^{-\alpha_\infty} \le |\hat{g}(\lambda)| \le C_2 |\lambda|^{-\alpha_\infty}.
	\end{equation*}
	\item \textbf{(H3) Low-frequency asymptotics (long time):} There exist an exponent $\alpha_0 \in [0,1)$ and constants $c_1, c_2 > 0$ such that, for $|\lambda|$ sufficiently small within $\mathbb{C} \setminus (-\infty, 0]$:
	\[ c_1 |\lambda|^{-\alpha_0} \le |\hat{g}(\lambda)| \le c_2 |\lambda|^{-\alpha_0}. \]
\end{itemize}

To illustrate the physical relevance of the admissible class and firmly ground our mathematical hypotheses in the rational mechanics of complex fluids, we highlight four canonical examples of anomalous momentum transport, emphasizing their spectral configurations.

\noindent \textbf{Example 1 (Pure fractional viscoelasticity).} The standard power-law fluid model, extensively validated in modern rheology for complex bulk and interfacial dynamics \cite{jaishankar2013}, is recovered by choosing the Riemann-Liouville fractional kernel, which yields the Laplace transform $\hat{g}(\lambda) = \lambda^{-\alpha}$ for $\alpha \in (0,1)$. This kernel satisfies (H1)--(H3) with exact scale invariance, wherein the high-frequency and low-frequency asymptotic exponents collapse to a single parameter $\alpha_\infty = \alpha_0 = \alpha$.

\noindent \textbf{Example 2 (Multi-term viscoelastic relaxation).} In complex viscoelastic fluids presenting structural heterogeneities (e.g., polymer melts), the momentum transport mechanism often transitions between distinct anomalous regimes over time. These dynamics can be modeled by a sum of fractional kernels, an approach systematically detailed in the theory of linear viscoelasticity \cite{mainardi2010}, yielding $\hat{g}(\lambda) = c_1 \lambda^{-\alpha} + c_2 \lambda^{-\beta}$ with $0 < \beta < \alpha < 1$. As $|\lambda| \to \infty$, the fluid response is dominated by the most singular exponent $\beta$, yielding $\alpha_\infty = \beta$, whereas as $|\lambda| \to 0$, the $\alpha$ term dominates the terminal relaxation, leading to $\alpha_0 = \alpha$.

\noindent \textbf{Example 3 (Fractional retardation and Cole-Cole relaxation).} Viscoelastic systems exhibiting a transition from an initial anomalous elastic state to a terminal Newtonian viscous profile at long times are modeled via fractional retardation kernels, an approach widely established in macroscopic fractional rheology \cite{mainardi2010}. The algebraic structure in the Laplace domain is given by $\hat{g}(\lambda) = (\lambda^\alpha + \gamma)^{-1}$ for $\alpha \in (0,1)$ and $\gamma > 0$. As $|\lambda| \to \infty$, the scaling $|\hat{g}(\lambda)| \sim |\lambda|^{-\alpha}$ establishes the short-time fractional exponent $\alpha_\infty = \alpha$. Conversely, as $|\lambda| \to 0$, the low-frequency transform behaves as $\gamma^{-1}|\lambda|^0$, definitively confirming the classical Newtonian crossover exponent $\alpha_0 = 0$.

\noindent \textbf{Example 4 (Multi-scale Prabhakar memory).} Momentum transport in macroscopically disordered fluid networks (such as dense micellar solutions) often transitions between two distinct anomalous regimes before reaching terminal saturation. Such behavior is captured by multi-scale kernels whose framework is governed by the generalized Mittag-Leffler functions of Prabhakar \cite{giusti2020, prabhakar1971}. The transform is defined as $\hat{g}(\lambda) = (\lambda^{\alpha_0} + \tau \lambda^{\alpha_\infty})^{-1}$ with $0 < \alpha_0 < \alpha_\infty < 1$ and $\tau > 0$. For short times ($|\lambda| \to \infty$), the high-frequency elastic response is dominated by the higher power $\tau \lambda^{\alpha_\infty}$, yielding the algebraic control $|\hat{g}(\lambda)| \sim \tau^{-1}|\lambda|^{-\alpha_\infty}$, confirming (H2). For long times ($|\lambda| \to 0$), the low-frequency relaxation is dictated by the lower power $\lambda^{\alpha_0}$, leading to $|\hat{g}(\lambda)| \sim |\lambda|^{-\alpha_0}$, confirming (H3).

\subsection{Mild framework}

Transitioning from the strong integro-differential formulation to an integral representation via the Duhamel principle circumvents classical differentiability requirements, thereby transferring the spatial and temporal derivatives directly onto the resolvent operator of the underlying linear non-local Stokes problem.

Applying the spatial Fourier transform (with frequency variable $\xi \in \mathbb{R}^N$) alongside the temporal Laplace transform (with dual variable $\lambda \in \mathbb{C}$) diagonalizes the spatial Laplacian and algebraically deconvolves the memory kernel, yielding the resolvent symbol in the joint frequency-Laplace domain
\begin{equation*}
	\widehat{S}(\lambda, \xi) = \frac{1}{\lambda + |\xi|^2 \hat{g}(\lambda)}.
\end{equation*}
The linear hereditary resolvent operator, denoted by $S(t)$, is then formally recovered via the inverse Laplace transform (the Bromwich integral) along a suitable contour $\Gamma$ in the complex plane, defining the spatial Fourier multiplier
\begin{equation*}
	\widehat{S}(t, \xi) = \frac{1}{2\pi i} \int_{\Gamma} e^{\lambda t} \frac{1}{\lambda + |\xi|^2 \hat{g}(\lambda)} \,d\lambda.
\end{equation*}

Treating the nonlinear convective term $\mathbb{P}\nabla \cdot (u \otimes u)$ as an external force field allows for the application of the variation of constants formula, thereby establishing the rigorous framework for mild solutions.

\begin{de}\label{def:mild_solution}
	Let $T > 0$ and $p \in (1, \infty)$. Given initial data $u_0 \in L^p(\mathbb{R}^N)$ satisfying the distributional divergence-free condition $\nabla \cdot u_0 = 0$, a measurable vector field $u: \mathbb{R}^N \times [0, T] \to \mathbb{R}^N$ is called a local-in-time mild solution to the NSHV equations on $[0, T]$ if it satisfies the following conditions:
	\begin{enumerate}
		\item[(i)] \textbf{Regularity:} $u \in C([0, T]; L^p(\mathbb{R}^N))$;
		\item[(ii)] \textbf{Incompressibility:} $\nabla \cdot u(t) = 0$ in the sense of distributions for all $t \in [0, T]$;
		\item[(iii)] \textbf{Integral Equation:} The function $u$ satisfies the Duhamel integral formulation
		\begin{equation}\label{eq:mild_formulation}
			u(t) = S(t)u_0 - \int_0^t \nabla S(t-s) \mathbb{P}(u \otimes u)(s) \,ds, \quad \text{for all } t \in [0, T],
		\end{equation}
		where the time integral converges absolutely in the Bochner sense in $L^p(\mathbb{R}^N)$.
	\end{enumerate}
	If the above conditions hold for every $T > 0$, $u$ is called a global-in-time mild solution.
\end{de}
	
\section{Preliminaries: Littlewood-Paley theory and function spaces}

In this section, we briefly review the foundational elements of the Littlewood-Paley mathematical framework, homogeneous Besov spaces, and Bony's para-differential calculus. These harmonic analysis tools are critical for establishing estimates for the non-local dual-scale resolvent operator and for decomposing the non-linear convective term. For a comprehensive treatise, we refer the reader to Grafakos \cite{grafakos2014} and Sawano \cite{sawano2018}.

\subsection{Dyadic partition of unity and Littlewood-Paley operators}
Let $\mathcal{C} = \{ \xi \in \mathbb{R}^N : \frac{3}{4} \le |\xi| \le \frac{8}{3} \}$ denote the standard dyadic annulus, and let $\mathcal{B} = \{ \xi \in \mathbb{R}^N : |\xi| \le \frac{4}{3} \}$ denote the corresponding low-frequency ball. We select smooth, radially symmetric cut-off functions $\chi \in C_c^\infty(\mathcal{B})$ and $\varphi \in C_c^\infty(\mathcal{C})$ with values in $[0,1]$ satisfying the spectral partition identities
\begin{equation*}
	\chi(\xi) + \sum_{j \ge 0} \varphi(2^{-j}\xi) = 1, \quad \text{and} \quad \sum_{j \in \mathbb{Z}} \varphi(2^{-j}\xi) = 1, \quad \forall \xi \neq 0.
\end{equation*}

For any $j \in \mathbb{Z}$, the homogeneous dyadic blocks $\Delta_j$ and the low-frequency cut-off operators $S_j$ are defined as frequency-localized Fourier multipliers. We denote by $\mathcal{S}'_h(\mathbb{R}^N)$ the space of tempered distributions with vanishing low-frequency components, rigorously defined as the subspace of $u \in \mathcal{S}'(\mathbb{R}^N)$ such that $\|S_j u\|_{L^\infty} \to 0$ as $j \to -\infty$. For any $u \in \mathcal{S}'_h(\mathbb{R}^N)$, we define
\begin{equation*}
	\Delta_j u = \mathcal{F}^{-1}\left( \varphi(2^{-j}\cdot) \widehat{u} \right) = \psi_j * u, \quad \text{where } \psi_j(x) = 2^{jN}\mathcal{F}^{-1}\{\varphi\}(2^jx);
\end{equation*}
\begin{equation*}
	S_j u = \sum_{k \le j-1} \Delta_k u = \mathcal{F}^{-1}\left( \chi(2^{-j}\cdot) \widehat{u} \right) = \phi_j * u, \quad \text{where } \phi_j(x) = 2^{jN}\mathcal{F}^{-1}\{\chi\}(2^jx).
\end{equation*}
In the distributional sense, any $u \in \mathcal{S}'_h(\mathbb{R}^N)$ can be recovered via the homogeneous Littlewood-Paley decomposition $u = \sum_{j \in \mathbb{Z}} \Delta_j u$.

\subsection{Homogeneous Besov spaces}
Equipped with the dyadic partition, we formalize the topology of homogeneous Besov spaces, which generalize standard Sobolev frameworks and precisely capture fractional regularity.

\begin{de}[Homogeneous Besov spaces \cite{sawano2018}]\label{def:besov_space}
	Let $s \in \mathbb{R}$ and $1 \le p, q \le \infty$. The homogeneous Besov space $\dot{B}^s_{p,q}(\mathbb{R}^N)$ is defined as the set of all tempered distributions $u \in \mathcal{S}'_h(\mathbb{R}^N)$ such that the norm
	\begin{equation*}
		\left\|u\right\|_{\dot{B}^s_{p,q}} = \left( \sum_{j \in \mathbb{Z}} 2^{jsq} \|\Delta_j u\|_{L^p}^q \right)^{1/q}
	\end{equation*}
	is finite, with the standard supremum modification for $q = \infty$:
	\begin{equation*}
		\left\|u\right\|_{\dot{B}^s_{p,\infty}} = \sup_{j \in \mathbb{Z}} 2^{js} \|\Delta_j u\|_{L^p}.
	\end{equation*}
\end{de}
The separable little Besov space, denoted by $\dot{b}^s_{p,\infty}(\mathbb{R}^N)$, is constructed as the topological closure within the $\dot{B}^s_{p,\infty}(\mathbb{R}^N)$ norm of the subspace $\mathcal{S}_0(\mathbb{R}^N)$ comprising Schwartz functions whose Fourier transforms are compactly supported away from the origin. It is uniquely characterized by the asymptotic adherence condition $\lim_{|j| \to \infty} 2^{js}\|\Delta_j u\|_{L^p} = 0$.

\subsection{Bernstein inequalities}
The localization of frequencies within dyadic spheres or annuli imposes strict constraints on spatial derivatives, allowing differential operators to be bounded by algebraic frequency prefactors.

\begin{lema}[Bernstein inequalities \cite{sawano2018}]\label{lem:bernstein}
	Let $1 \le p \le r \le \infty$. There exists a uniform constant $C > 0$, independent of $j$, such that for any multi-index $\beta \in \mathbb{N}^0_N$ and any $j \in \mathbb{Z}$, the following assertions hold:
	\begin{enumerate}
		\item[(i)] For spectral balls: If $\textnormal{supp}(\widehat{f}) \subset B(0, R 2^j)$, then
		\begin{equation*}
			\|\partial^\beta f\|_{L^p} \le C^{|\beta|+1} 2^{j|\beta|} \|f\|_{L^p}.
		\end{equation*}
		\item[(ii)] For spectral annuli: If $\textnormal{supp}(\widehat{f}) \subset \mathcal{C}(0, R_1 2^j, R_2 2^j)$, then
		\begin{equation*}
			C^{-|\beta|-1} 2^{j|\beta|} \|f\|_{L^p} \le \|\partial^\beta f\|_{L^p} \le C^{|\beta|+1} 2^{j|\beta|} \|f\|_{L^p}.
		\end{equation*}
		\item[(iii)] Embedding constraints: If $\textnormal{supp}(\widehat{f}) \subset B(0, R 2^j)$, then
		\begin{equation*}
			\|f\|_{L^r} \le C 2^{jN\left(\frac{1}{p} - \frac{1}{r}\right)} \|f\|_{L^p}.
		\end{equation*}
	\end{enumerate}
\end{lema}

\subsection{Bony's paraproduct decomposition}\label{paraproduct}
The pointwise multiplication of two tempered distributions necessitates a rigorous analysis of overlapping frequency supports. Bony's para-differential calculus \cite{bahouri2011, bony1981, sawano2018} decouples the product $uv$ into directional non-local operators, filtering high-frequency modulations systematically.

Let $u, v \in \mathcal{S}'_h(\mathbb{R}^N)$. The product can be formally expanded using the resolution of the identity as $uv = \sum_{k, m \in \mathbb{Z}} \Delta_k u \Delta_m v$. Grouping these terms according to their relative frequency positions yields Bony's decomposition
\begin{equation*}
	uv = T_u v + T_v u + R(u,v),
\end{equation*}
where the operators are explicitly defined as:
\begin{enumerate}
	\item[(i)] The paraproduct of $v$ by $u$, denoted by $T_u v$, gathers low frequencies of $u$ interacting with high frequencies of $v$:
	\begin{equation*}
		T_u v = \sum_{k \in \mathbb{Z}} S_{k-1} u \Delta_k v.
	\end{equation*}
	\item[(ii)] The symmetric paraproduct of $u$ by $v$: $T_v u = \sum_{k \in \mathbb{Z}} S_{k-1} v \Delta_k u$.
	\item[(iii)] The high-high symmetric residue operator $R(u,v)$ captures close-range frequency interactions where neither factor asymptotically dominates
	\begin{equation*}
		R(u,v) = \sum_{k \in \mathbb{Z}} \sum_{|k-m| \le 1} \Delta_k u \Delta_m v.
	\end{equation*}
\end{enumerate}
The compact support of the generating functions $\chi$ and $\varphi$ ensures that the $j$-th dyadic block of these interactions is finitely supported:
\begin{equation*}
	\Delta_j(T_u v) = \sum_{|j-k| \le 4} \Delta_j(S_{k-1}u \Delta_k v), \quad \text{and} \quad \Delta_j R(u,v) = \sum_{k \ge j-3} \Delta_j(\Delta_k u \widetilde{\Delta}_k v),
\end{equation*}
where $\widetilde{\Delta}_k = \Delta_{k-1} + \Delta_k + \Delta_{k+1}$.

\subsection{Fourier multipliers and the Leray-Hopf projector}\label{Fourier multipliers and the Leray-Hopf projector}
To resolve the pressure term in the fluid dynamics system, we apply the Leray-Hopf projector $\mathbb{P}$ onto the divergence-free vector fields. In Fourier space, the matrix symbol of this projection is
\[\sigma(\mathbb{P})_{mn}(\xi) = \delta_{mn} - \xi_m\xi_n/|\xi|^2.\]

Since the symbol is homogeneous of degree zero and infinitely differentiable away from the origin, classical Calder\'on-Zygmund theory guarantees that $\mathbb{P}$ acts as a bounded linear operator on $L^p(\mathbb{R}^N)$ for all $1 < p < \infty$ \cite{lemarie2002}. However, such singular integrals generally fail to be bounded in the limits $L^1(\mathbb{R}^N)$ and $L^\infty(\mathbb{R}^N)$. To circumvent this topological obstruction, we exploit dyadic localization. When restricted to the annulus $\mathcal{C}_j$, the localized symbol $\sigma(\mathbb{P})(\xi) \varphi(2^{-j}\xi)$ is smooth and compactly supported. By Paley-Wiener theory and the foundational properties of Littlewood-Paley multipliers \cite[Chapter 6]{grafakos2014}, its inverse Fourier transform defines a Schwartz convolution kernel $K_j \in \mathcal{S}(\mathbb{R}^N) \subset L^1(\mathbb{R}^N)$. Furthermore, the scaling relation $K_j(x) = 2^{jN} K_0(2^j x)$ ensures that its $L^1$-norm is strictly scale-invariant ($\|K_j\|_{L^1} = \|K_0\|_{L^1}$). By Young's convolution inequality, the composite operator $\mathbb{P}\Delta_j$ maps $L^\infty(\mathbb{R}^N)$ to $L^\infty(\mathbb{R}^N)$ uniformly with respect to $j \in \mathbb{Z}$, a critical property extensively utilized in critical Besov spaces.
	
	\section{Linear estimates}
	
	This section establishes the $L^q - L^p$ decay estimates for the linear hereditary resolvent operator $S(t)$ and its composition with the spatial gradient operator, denoted by $\nabla S(t)$. In contrast to the classical heat semigroup generated by the standard Laplacian, the underlying linear dynamics are governed by a non-local memory kernel that imposes a spectral dichotomy upon the resolvent operator. Consequently, its regularizing capacity transitions from a short-time regime—dictated by the high-frequency parameter $\alpha_\infty \in (0,1)$—to a long-time asymptotic regime governed by the low-frequency parameter $\alpha_0 \in [0,1)$.
	
	Capturing this anomalous dissipation necessitates defining the respective local and global scaling exponents. For the short-time high-frequency regime, we set
	\begin{equation*}
		\gamma_{q,p} = \frac{1+\alpha_\infty}{2} N \left(\frac{1}{q} - \frac{1}{p}\right), \quad \text{and} \quad \tilde{\gamma}_{q,p} = \gamma_{q,p} + \frac{1+\alpha_\infty}{2}.
	\end{equation*}
	For the long-time low-frequency regime, the exponents are given by
	\begin{equation*}
		\beta_{q,p} = \frac{1+\alpha_0}{2} N \left(\frac{1}{q} - \frac{1}{p}\right), \quad \text{and} \quad \tilde{\beta}_{q,p} = \beta_{q,p} + \frac{1+\alpha_0}{2}.
	\end{equation*}
	
	The ensuing mapping properties, categorized into short- and long-time regimes, naturally encode the topological restrictions imposed by both the Leray-Hopf projector and the fundamental geometric integrability thresholds.
	
	\begin{lema}[Short-time linear estimates]\label{lem:short_time_estimates}
		Let $T > 0$. Under the hypotheses (H1) and (H2), there exists a constant $C_T > 0$ such that for any $0 < t \le T$, the following estimates hold:
		\begin{enumerate}
			\item[(i)] Pure resolvent: Let $q \in [1, \infty)$ and $p \in [q, \infty]$ satisfy the condition $\frac{1}{q} - \frac{1}{p} < \frac{2}{N}$. For any vector field $f \in L^q(\mathbb{R}^N)$, we have
			\begin{equation*}
				\|S(t)f\|_{L^p} \le C_T t^{-\gamma_{q,p}} \|f\|_{L^q}.
			\end{equation*}
			\item[(ii)] Gradient with projector: Let $q \in (1, \infty)$ and $p \in [q, \infty]$ satisfy the condition $\frac{1}{q} - \frac{1}{p} < \frac{1}{N}$. For any vector field $f \in L^q(\mathbb{R}^N)$, we have
			\begin{equation*}
				\|\nabla S(t) \mathbb{P} f\|_{L^p} \le C_T t^{-\tilde{\gamma}_{q,p}} \|f\|_{L^q}.
			\end{equation*}
		\end{enumerate}
	\end{lema}
	
	\begin{proof}
		The pure resolvent estimate (i) was rigorously established in \cite{deandrade2026} by analyzing the Bromwich integral of the multiplier over the high-frequency spectrum; hence, our focus remains strictly on (ii).
		
Given that the spatial gradient $\nabla$, the Leray-Hopf projector $\mathbb{P}$, and the resolvent operator $S(t)$ pairwise commute, their composition can be decoupled to act sequentially upon the vector field, yielding $\nabla S(t) \mathbb{P} f = \nabla S(t) (\mathbb{P} f)$. Rather than relying on ad-hoc physical kernel decompositions, we explicitly leverage the pseudo-differential framework established in \cite{deandrade2026}.

		For the short-time regime $t \in (0, T]$, the temporal scaling $\eta = \xi t^{\frac{1+\alpha_\infty}{2}}$ and the spatial dilation parameter $\mu = t^{\frac{1+\alpha_\infty}{2}}$ diagonalize the high-frequency operator. The scaled pure resolvent symbol $\tilde{\Psi}(t, \eta) = \widehat{S}(t, \mu^{-1}\eta)$ belongs uniformly to the H\"ormander class $S^{-2}_{1,0}$ \cite[Lemma 3.1]{deandrade2026}. Applying the spatial gradient $\nabla$ introduces the Fourier multiplier $i\xi$. Under the scaling transformation, this reads $i\xi = \mu^{-1} i\eta$. Thus, the gradient operator satisfies
		\begin{equation}\label{eq:grad_scaling}
			\nabla S(t)f(x) = \mu^{-1} \mathcal{F}^{-1}\big\{ i\eta \tilde{\Psi}(t, \eta) \widehat{f_{\mu}}(\eta) \big\}(\mu^{-1}x) = \mu^{-1} T_t^1(f_{\mu})(\mu^{-1}x),
		\end{equation}
		where $f_{\mu}(y) = f(\mu y)$, and $T_t^1$ is the pseudo-differential operator associated with the multiplier $m_1(t, \eta) = i\eta \tilde{\Psi}(t, \eta)$.
		
		Since $\tilde{\Psi} \in S^{-2}_{1,0}$, the gradient symbol $m_1$ structurally belongs to $S^{-1}_{1,0}$. By classical pseudo-differential theory, a symbol in the class $S^{-1}_{1,0}$ generates a spatial convolution kernel $G_t^1(w)$ that is smooth away from the origin while encapsulating a localized singularity of order $\mathcal{O}(|w|^{-(N-1)})$ near the origin. Consequently, the local integrability condition $G_t^1 \in L^r(\mathbb{R}^N)$ dictates that $\int_{|w|<1} |w|^{-r(N-1)} \,dw < \infty$, yielding the geometric constraint $r < \frac{N}{N-1}$. Utilizing Young's convolution inequality, the operator $T_t^1$ acts boundedly from $L^q(\mathbb{R}^N)$ to $L^p(\mathbb{R}^N)$ under the relation $1 + 1/p = 1/r + 1/q$. This mapping property explicitly translates the geometric singularity at the origin into the fundamental topological barrier
		\begin{equation*}
			\frac{1}{q} - \frac{1}{p} = 1 - \frac{1}{r} < 1 - \frac{N-1}{N} = \frac{1}{N}.
		\end{equation*}
		
		Under this geometric constraint, evaluating the $L^p$-norm of \eqref{eq:grad_scaling} yields
		\begin{equation*}
			\|\nabla S(t) f\|_{L^p} = \mu^{-1} \mu^{\frac{N}{p}} \|T_t^1(f_{\mu})\|_{L^p} \le C \mu^{-1 + \frac{N}{p}} \|f_{\mu}\|_{L^q} = C \mu^{-1 + N\left(\frac{1}{p} - \frac{1}{q}\right)} \|f\|_{L^q}.
		\end{equation*}
		Substituting the scaling parameter $\mu = t^{\frac{1+\alpha_\infty}{2}}$ explicitly recovers the targeted temporal decay exponent
		\begin{equation*}
			t^{\frac{1+\alpha_\infty}{2} \left[ -1 - N\left(\frac{1}{q} - \frac{1}{p}\right) \right]} = t^{-\frac{1+\alpha_\infty}{2} - \gamma_{q,p}} = t^{-\tilde{\gamma}_{q,p}}.
		\end{equation*}
		Finally, since the Leray-Hopf projector $\mathbb{P}$ is uniformly bounded in $L^q(\mathbb{R}^N)$ for all $q \in (1, \infty)$, the desired composite bounds follow immediately, completing the proof.
	\end{proof}
	
	\begin{lema}[Long-time linear estimates]\label{lem:long_time_estimates}
		Under the hypotheses (H1), (H2), and (H3), there exists a constant $C > 0$ such that for sufficiently large times $t > T$, the following asymptotic bounds hold:
		\begin{enumerate}
			\item[(i)] Pure resolvent: Let $q \in [1, \infty)$ and $p \in [q, \infty]$ satisfy the condition $\frac{1}{q} - \frac{1}{p} < \frac{2}{N}$. For any vector field $f \in L^q(\mathbb{R}^N)$, we have
			\begin{equation*}
				\|S(t)f\|_{L^p} \le C t^{-\beta_{q,p}} \|f\|_{L^q}.
			\end{equation*}
			\item[(ii)] Gradient with projector: Let $q \in (1, \infty)$ and $p \in [q, \infty]$ satisfy the condition $\frac{1}{q} - \frac{1}{p} < \frac{1}{N}$. For any vector field $f \in L^q(\mathbb{R}^N)$, we have
			\begin{equation*}
				\|\nabla S(t) \mathbb{P} f\|_{L^p} \le C t^{-\tilde{\beta}_{q,p}} \|f\|_{L^q}.
			\end{equation*}
		\end{enumerate}
	\end{lema}
	
	\begin{proof}
		Analogous to the short-time regime, the pure resolvent asymptotic estimate (i) follows directly from the low-frequency dyadic framework established in \cite[Lemma 3.4]{deandrade2026} under hypotheses (H1)--(H3); therefore, we restrict our attention to the gradient estimate (ii).
		
		Because the generalized memory kernel $g$ breaks exact global scale invariance, its  limit cannot be deduced from isolated self-similar profiles. Instead, we rely on the low-frequency Littlewood-Paley dyadic partition $S(t) = S_0(t) + \sum_{j=1}^\infty S_j(t)$ constructed in \cite{deandrade2026}. For the high-frequency residual branch $S_0(t)$, the roots of the characteristic equation are strictly bounded away from the imaginary axis, establishing a uniform spectral gap $\text{Re}(\lambda) \le -\tilde{c} < 0$. By deforming the Bromwich contour strictly into the left half-plane, the corresponding resolvent component acquires an exponential damping $\mathcal{O}(e^{-\tilde{c}t})$. Although the spatial gradient introduces an additional algebraic factor, this exponential decay overwhelmingly dominates the long-time behavior, ensuring that the gradient contribution of $S_0(t)$ remains subordinate to the algebraic rate for $t \ge 1$.
		
		For the low-frequency components $j \ge 1$, the spatial frequencies are localized within the dyadic annuli $|\xi| \sim 2^{-j}$. Consequently, the action of the spatial gradient $\nabla$ geometrically extracts a spectral factor proportional to $2^{-j}$ (cf. Lemma \ref{lem:bernstein}). Applying this property to the uniform H\"ormander multiplier bounds derived in \cite{deandrade2026} yields
		\begin{equation*}
			\|\nabla S_j(t) f\|_{L^p} \le C 2^{-j} \cdot 2^{-jN\left(\frac{1}{q} - \frac{1}{p}\right)} e^{-ct 2^{-j\gamma_0}} \|\Delta_j f\|_{L^q} = C 2^{-j\left[N\left(\frac{1}{q} - \frac{1}{p}\right) + 1\right]} e^{-ct 2^{-j\gamma_0}} \|\Delta_j f\|_{L^q},
		\end{equation*}
		where $\gamma_0 = \frac{2}{1+\alpha_0}$ determines the speed of the moving poles toward the origin.
		
		To estimate the summation over $j \ge 1$ for large times $t \ge 1$, we bisect the series at the time-dependent critical index $j_0(t) \in \mathbb{N}$. This index is chosen such that the dynamic transition scale satisfies $t 2^{-j_0(t)\gamma_0} \sim 1$, which algebraically imposes $2^{-j_0(t)} \sim t^{-1/\gamma_0} = t^{-\frac{1+\alpha_0}{2}}$. Within the damped frequency scales $j \le j_0(t)$, the condition $t 2^{-j\gamma_0} \ge 1$ holds. Because the exponential damping term $e^{-ct 2^{-j\gamma_0}}$ decays for smaller indices $j$ (representing relatively higher frequencies), the sum is dominated by its largest term at the cutoff boundary $j = j_0(t)$:
		\begin{align*}
			\sum_{j=1}^{j_0(t)} \|\nabla S_j(t) f\|_{L^p} &\le C \|f\|_{L^q} \sum_{j=1}^{j_0(t)} 2^{-j\left[N\left(\frac{1}{q} - \frac{1}{p}\right) + 1\right]} e^{-ct 2^{-j\gamma_0}} \\
			&\le C \|f\|_{L^q} 2^{-j_0(t)\left[N\left(\frac{1}{q} - \frac{1}{p}\right) + 1\right]} \\
			&= C \|f\|_{L^q} \left( t^{-\frac{1+\alpha_0}{2}} \right)^{N\left(\frac{1}{q} - \frac{1}{p}\right) + 1} = C t^{-\beta_{q,p} - \frac{1+\alpha_0}{2}} \|f\|_{L^q} = C t^{-\tilde{\beta}_{q,p}} \|f\|_{L^q}.
		\end{align*}
		
		Conversely, over the ultra-low frequency scales $j > j_0(t)$, the temporal condition transitions to $t 2^{-j\gamma_0} < 1$, rendering the exponential damping factor of order $\mathcal{O}(1)$. Since the relation $q \le p$  guarantees that the algebraic exponent satisfies $N(1/q - 1/p) + 1 \ge 1 > 0$, the remaining summation transforms into a strictly convergent geometric series. This tail is thus bounded by its leading term at $j = j_0(t) + 1$:
		\begin{align*}
			\sum_{j=j_0(t)+1}^{\infty} \|\nabla S_j(t) f\|_{L^p} &\le C \|f\|_{L^q} \sum_{j=j_0(t)+1}^\infty 2^{-j\left[N\left(\frac{1}{q} - \frac{1}{p}\right) + 1\right]}\\
			& \le C 2^{-j_0(t)\left[N\left(\frac{1}{q} - \frac{1}{p}\right) + 1\right]} \|f\|_{L^q} \le C t^{-\tilde{\beta}_{q,p}} \|f\|_{L^q}.
		\end{align*}
		
		Synthesizing the bounds from both spectral domains and leveraging the uniform boundedness of $\mathbb{P}$ in $L^q(\mathbb{R}^N)$ secures the global asymptotic estimate $\|\nabla S(t) \mathbb{P} f\|_{L^p} \le C t^{-\tilde{\beta}_{q,p}} \|f\|_{L^q}$. This derivation closes the proof.
	\end{proof}
	
\section{Local well-posedness and the geometric threshold $p_c$}

Having established the linear decay estimates, we now address the local Hadamard well-posedness of the NSHV equations. The integral Duhamel formulation~\eqref{eq:mild_formulation} provides a natural framework to construct mild solutions via the Banach fixed-point theorem, simultaneously yielding existence, uniqueness, and continuous dependence on the initial data.

We define the bilinear integral form
\begin{equation*}
	B(u,v)(t) = \int_0^t \nabla S(t-s) \mathbb{P}(u \otimes v)(s) \,ds.
\end{equation*}
The mild formulation is thus rewritten as $u(t) = S(t)u_0 - B(u,u)(t)$. Applying the contraction mapping principle within $C([0,T]; L^p(\mathbb{R}^N))$ requires controlling the $L^p$-norm of this bilinear form. By invoking the short-time spatial derivative estimate from Lemma~\ref{lem:short_time_estimates}(ii), the non-local term imposes a strict integrability condition. Given vector fields $u,v \in L^p(\mathbb{R}^N)$, H\"older's inequality dictates that $u \otimes v \in L^{p/2}(\mathbb{R}^N)$. Applying the linear bound with the source space configured as $q = p/2$ yields
\begin{equation*}
	\|B(u,v)(t)\|_{L^p} \le C_T \int_0^t (t-s)^{-\tilde{\gamma}_{p/2,p}} \|u(s)\|_{L^p} \|v(s)\|_{L^p} \,ds.
\end{equation*}
The fractional temporal singularity is explicitly governed by the local exponent
\begin{equation*}
	\tilde{\gamma}_{p/2,p} = \frac{1+\alpha_\infty}{2} N \left(\frac{2}{p} - \frac{1}{p}\right) + \frac{1+\alpha_\infty}{2} = \frac{1+\alpha_\infty}{2} \left(\frac{N}{p} + 1\right).
\end{equation*}
Ensuring the absolute convergence of the temporal convolution near the singular limit $s \to t^-$ demands the strict bound $\tilde{\gamma}_{p/2,p} < 1$. The threshold at which this inequality degenerates into an equality precisely isolates the critical Lebesgue exponent
\begin{equation*}
	p_c = N\left(\frac{1+\alpha_\infty}{1-\alpha_\infty}\right).
\end{equation*}

Since the short-time material parameter satisfies $\alpha_\infty \in (0,1)$ and the spatial dimension is $N \ge 2$, the critical threshold satisfies $p_c > N \ge 2$. Consequently, assuming $p \ge p_c$ simultaneously guarantees $p/2 > 1$---securing the uniform boundedness of the Leray-Hopf projector $\mathbb{P}$ in $L^{p/2}(\mathbb{R}^N)$ via the Marcinkiewicz multiplier theorem---and the geometric integrability constraint $\frac{1}{p/2} - \frac{1}{p} < \frac{1}{N}$, ensuring the applicability of the spatial derivative bounds derived in Lemma~\ref{lem:short_time_estimates}.

\begin{remark}[Asymptotic limits of the critical threshold]\label{rem:alpha_dependence}
	The dependence of the critical Lebesgue exponent $p_c = N(\frac{1+\alpha_\infty}{1-\alpha_\infty})$ on the high-frequency memory parameter $\alpha_\infty \in (0,1)$ mathematically quantifies the delicate analytical balance between the non-local convective transport and the pseudo-differential dissipation. This dependence exhibits two fundamental asymptotic regimes:
	\begin{enumerate}
		\item[(i)] \textbf{The Newtonian limit ($\alpha_\infty \to 0$):} As the high-frequency fractional anomaly vanishes, the memory kernel functionally approaches instantaneous dissipation. Concurrently, the critical exponent converges to $\lim_{\alpha_\infty \to 0} p_c(\alpha_\infty) = N$. This convergence recovers the classical critical space $L^N(\mathbb{R}^N)$ established by Kato \cite{kato1984} for the standard Navier-Stokes equations, ensuring the topological consistency of our non-local framework.
		\item[(ii)] \textbf{The extreme elasticity limit ($\alpha_\infty \to 1^-$):} As the material memory intensifies, the initial temporal relaxation becomes highly singular, implying that the fluid retains a strong elastic response at short times. Analytically, the fractional temporal smoothing provided by the resolvent deteriorates. To counterbalance the unabated spatial loss of regularity induced by the convective derivative, the requisite spatial integrability diverges: $\lim_{\alpha_\infty \to 1^-} p_c(\alpha_\infty) = +\infty$. In this regime, the functional space must be extremely smooth to absorb the convective cascade, thereby severely restricting the admissible initial data.
	\end{enumerate}
\end{remark}

\subsection{Local well-posedness in the critical and subcritical regimes}
This temporal integrability boundary divides the existence theory into two distinct regimes: the integrable subcritical regime ($p > p_c$) and the scale-critical regime ($p = p_c$). We formalize this dichotomy in the following theorem.

\begin{theorem}[Local well-posedness and the threshold $p_c$]\label{thm:local_well_posedness}
	Let $N \ge 2$ and assume the structural hypotheses (H1) and (H2) hold. The critical Lebesgue threshold $p_c = N(\frac{1+\alpha_\infty}{1-\alpha_\infty})$ dictates the well-posedness limits for the NSHV equations:
	\begin{enumerate}
		\item[(i)] Subcritical regime ($p > p_c$): For any initial data $u_0 \in L^p(\mathbb{R}^N)$ satisfying $\nabla \cdot u_0 = 0$ distributionally, there exists a maximal time $T_{\text{max}} > 0$ and a unique local-in-time mild solution $u \in C([0, T_{\text{max}}); L^p(\mathbb{R}^N))$ to problem~\eqref{eq:NSHV_projected}.
		\item[(ii)] Critical regime ($p = p_c$): If in addition (H3) holds, then there exists a small constant $\varepsilon > 0$ such that, for any initial data $u_0 \in L^{p_c}(\mathbb{R}^N)$ satisfying $\nabla \cdot u_0 = 0$ distributionally and $\|u_0\|_{L^{p_c}} < \varepsilon$, problem~\eqref{eq:NSHV_projected} admits a unique global-in-time mild solution satisfying $\lim_{t \to 0^+} t^{\gamma_{p_c,2p_c}} \|u(t)\|_{L^{2p_c}} = 0$.
	\end{enumerate}
	Furthermore, in both cases, the data-to-solution map $u_0 \mapsto u$ is locally Lipschitz continuous, securing Hadamard well-posedness in their respective functional spaces.
\end{theorem}

\begin{proof}
	We divide the proof according to the respective regimes.
	
	\noindent\textbf{Proof of (i) Subcritical regime ($p > p_c$).}
	Let $T > 0$ and define the Banach space
	\[
	\mathcal{X}_T = C([0,T]; L^p(\mathbb{R}^N))
	\]
	equipped with the standard supremum norm $\|u\|_{\mathcal{X}_T} = \sup_{0 \le t \le T} \|u(t)\|_{L^p}$. Seeking a unique fixed point for the nonlinear mapping
	\[
	\Phi(u)(t) = S(t)u_0 - B(u,u)(t)
	\]
	in $\mathcal{X}_T$, we observe that since $p > p_c$, the geometric exponent satisfies $\tilde{\gamma}_{p/2,p} < 1$. By Lemma~\ref{lem:short_time_estimates}, the linear term $S(\cdot)u_0$ belongs to $\mathcal{X}_T$ and satisfies the uniform bound $\|S(\cdot)u_0\|_{\mathcal{X}_T} \le C_T \|u_0\|_{L^p}$. Applying the spatial derivative estimate with $q = p/2$ to the bilinear form yields, for any $u, v \in \mathcal{X}_T$:
	\begin{align*}
		\|B(u,v)(t)\|_{L^p} &\le C_T \int_0^t (t-s)^{-\tilde{\gamma}_{p/2,p}} \|u(s)\|_{L^p} \|v(s)\|_{L^p} \,ds \\
		&\le C_T \|u\|_{\mathcal{X}_T} \|v\|_{\mathcal{X}_T} \int_0^t (t-s)^{-\tilde{\gamma}_{p/2,p}} \,ds \\
		&\le C_T \frac{T^{1-\tilde{\gamma}_{p/2,p}}}{1-\tilde{\gamma}_{p/2,p}} \|u\|_{\mathcal{X}_T} \|v\|_{\mathcal{X}_T}.
	\end{align*}
	
	Verifying $\Phi(u) \in \mathcal{X}_T$ requires establishing the strong continuity of the mapping $t \mapsto \Phi(u)(t)$ in $L^p(\mathbb{R}^N)$. The strong continuity of the linear term $t \mapsto S(t)u_0$ down to $t=0$ follows from the sectorial properties of the resolvent symbol. For the bilinear term, let $0 \le t_1 < t_2 \le T$. We decompose the increment as
	\begin{align*}
		B(u,v)(t_2) - B(u,v)(t_1) &= \int_{t_1}^{t_2} \nabla S(t_2-s) \mathbb{P}(u \otimes v)(s) \,ds \\
		&\quad + \int_0^{t_1} \left[ \nabla S(t_2-s) - \nabla S(t_1-s) \right] \mathbb{P}(u \otimes v)(s) \,ds \\
		&=: I_1 + I_2.
	\end{align*}
	For $I_1$, utilizing the uniform boundedness of $u, v \in \mathcal{X}_T$, we estimate
	\[
	\|I_1\|_{L^p} \le C_T \|u\|_{\mathcal{X}_T} \|v\|_{\mathcal{X}_T} \int_{t_1}^{t_2} (t_2-s)^{-\tilde{\gamma}_{p/2,p}} \,ds = C_T \|u\|_{\mathcal{X}_T} \|v\|_{\mathcal{X}_T} \frac{(t_2-t_1)^{1-\tilde{\gamma}_{p/2,p}}}{1-\tilde{\gamma}_{p/2,p}},
	\]
	which vanishes as $t_2 \searrow t_1$ since $1-\tilde{\gamma}_{p/2,p} > 0$. For $I_2$, the analyticity of the multiplier in the sector $\Sigma_{\theta_\infty}$ implies that the map $t \mapsto \nabla S(t)$ is strongly continuous for $t > 0$. Thus, the integrand converges to zero in $L^p(\mathbb{R}^N)$ for almost every $s \in (0, t_1)$. Moreover, it is majorized by the integrable function $2C_T (t_1-s)^{-\tilde{\gamma}_{p/2,p}} \|u\|_{\mathcal{X}_T} \|v\|_{\mathcal{X}_T}$. By the Lebesgue Dominated Convergence Theorem, $\|I_2\|_{L^p} \to 0$. This confirms that $B(u,v) \in \mathcal{X}_T$. Setting the constant $K_T = C_T (1-\tilde{\gamma}_{p/2,p})^{-1} T^{1-\tilde{\gamma}_{p/2,p}}$, we conclude that $\Phi$ is well-defined and maps $\mathcal{X}_T$ into itself.
	
	To invoke the contraction mapping principle, we define the closed absorbing ball
	\[\mathcal{B}_R = \{ u \in \mathcal{X}_T : \|u\|_{\mathcal{X}_T} \le R \}\]
	with radius $R = 2C_T\|u_0\|_{L^p}$. For any $u \in \mathcal{B}_R$, the triangle inequality yields
	\begin{equation*}
		\|\Phi(u)\|_{\mathcal{X}_T} \le \|S(\cdot)u_0\|_{\mathcal{X}_T} + \|B(u,u)\|_{\mathcal{X}_T} \le \frac{R}{2} + K_T R^2.
	\end{equation*}
	Furthermore, utilizing the algebraic identity $u \otimes u - v \otimes v = u \otimes (u-v) + (u-v) \otimes v$ and the bilinearity of $B$, we evaluate the difference for any $u,v \in \mathcal{B}_R$ as
	\begin{equation*}
		\|\Phi(u) - \Phi(v)\|_{\mathcal{X}_T} \le \|B(u, u-v)\|_{\mathcal{X}_T} + \|B(u-v, v)\|_{\mathcal{X}_T} \le 2 K_T R \|u-v\|_{\mathcal{X}_T}.
	\end{equation*}
	To ensure that $\Phi$ leaves $\mathcal{B}_R$ invariant and acts as a strict contraction, we impose the condition $2 K_T R \le 1/2$. Substituting the definitions of $K_T$ and $R$, this requirement is equivalent to
	\begin{equation}\label{eq:T_condition}
		T^{1-\tilde{\gamma}_{p/2,p}} \le \frac{1-\tilde{\gamma}_{p/2,p}}{8 C_T^2 \|u_0\|_{L^p}}.
	\end{equation}
	Since $1-\tilde{\gamma}_{p/2,p} > 0$, the left-hand side approaches zero as $T \to 0^+$. Therefore, for any given initial data $u_0 \in L^p(\mathbb{R}^N)$, there exists a sufficiently short time $T > 0$ such that \eqref{eq:T_condition} holds. By the Banach fixed-point theorem, $\Phi$ possesses a unique fixed point in $\mathcal{B}_R$, establishing local existence.
	
	To extend this local uniqueness from the ball $\mathcal{B}_R$ to the entire space $\mathcal{X}_T$, suppose $u, v \in \mathcal{X}_T$ are two mild solutions originating from the same initial data $u_0$, and let $w = u - v$. By the bilinearity of $B$, their difference satisfies the integral equation
	\[w(t) = -B(u, w)(t) - B(w, v)(t).\]
	Given that $u$ and $v$ belong to $C([0,T]; L^p(\mathbb{R}^N))$, their norms are uniformly bounded on $[0,T]$ by some constant $M > 0$. Estimating the $L^p$-norm of $w(t)$ yields
	\begin{equation*}
		\|w(t)\|_{L^p} \le C_T M \int_0^t (t-s)^{-\tilde{\gamma}_{p/2,p}} \|w(s)\|_{L^p} \,ds.
	\end{equation*}
	Since the exponent satisfies $\tilde{\gamma}_{p/2,p} < 1$, we can invoke the singular Gr\"onwall inequality (see Henry \cite{henry1981}). As the homogeneous driving term is zero, the inequality mandates $\|w(t)\|_{L^p} = 0$ for all $t \in [0,T]$, proving that $u \equiv v$ globally in $\mathcal{X}_T$.
	
	Now, we address the maximal time of existence. The local solution constructed on $[0, T]$ can be evaluated at the endpoint to supply new initial data $u(T) \in L^p(\mathbb{R}^N)$, allowing the solution to be extended to a larger interval $[0, T']$. Since the admissible step size depends strictly on the magnitude of the $L^p$-norm via \eqref{eq:T_condition}, we can iterate this continuation argument to generate a maximal interval of existence $[0, T_{\text{max}})$. If $T_{\text{max}} < \infty$, the local existence bounds dictate that the $L^p$-norm must grow unbounded; otherwise, a uniform time step would permit a trivial extension beyond $T_{\text{max}}$. This confirms the standard blow-up alternative: if $T_{\text{max}} < \infty$, then necessarily $\limsup_{t \nearrow T_{\text{max}}} \|u(t)\|_{L^p} = \infty$.
	
	Finally, to establish continuous dependence on the initial data within the subcritical regime, let $u, v \in \mathcal{B}_R$ be two mild solutions originating from divergence-free initial data $u_0, v_0 \in L^p(\mathbb{R}^N)$, respectively. Utilizing the algebraic identity for the bilinear form, their difference satisfies
	\begin{equation*}
		u(t) - v(t) = S(t)(u_0 - v_0) - B(u, u-v)(t) - B(u-v, v)(t).
	\end{equation*}
	Taking the norm in the Banach space $\mathcal{X}_T$ and applying the previously derived bilinear estimates yields
	\begin{align*}
		\|u - v\|_{\mathcal{X}_T} &\le \|S(\cdot)(u_0 - v_0)\|_{\mathcal{X}_T} + \|B(u, u-v)\|_{\mathcal{X}_T} + \|B(u-v, v)\|_{\mathcal{X}_T} \\
		&\le C_T \|u_0 - v_0\|_{L^p} + K_T \left( \|u\|_{\mathcal{X}_T} + \|v\|_{\mathcal{X}_T} \right) \|u - v\|_{\mathcal{X}_T}.
	\end{align*}
	Since both solutions reside within the closed ball $\mathcal{B}_R$, we have $\|u\|_{\mathcal{X}_T} + \|v\|_{\mathcal{X}_T} \le 2R$. Recalling the sharp contraction condition~\eqref{eq:T_condition}, which inherently enforces $2 K_T R \le 1/2$, the estimate reduces to
	\begin{equation*}
		\|u - v\|_{\mathcal{X}_T} \le C_T \|u_0 - v_0\|_{L^p} + \frac{1}{2} \|u - v\|_{\mathcal{X}_T}.
	\end{equation*}
	Absorbing the non-linear residual into the left-hand side immediately generates the local Lipschitz continuity bound for the data-to-solution map
	\begin{equation*}
		\|u - v\|_{\mathcal{X}_T} \le 2 C_T \|u_0 - v_0\|_{L^p},
	\end{equation*}
	which successfully completes the verification of local well-posedness in the sense of Hadamard for the subcritical regime.
	
	\noindent\textbf{Proof of (ii) Critical regime ($p = p_c$).}
	Reaching the geometric threshold $p = p_c$ triggers an analytical obstruction, as the exponent evaluates precisely to $\tilde{\gamma}_{p_c/2, p_c} = 1$. This structural barrier causes the temporal integration kernel to diverge logarithmically, rendering the unweighted space $C([0,T]; L^{p_c}(\mathbb{R}^N))$ structurally inadequate to close the fixed-point argument. Circumventing this critical logarithmic collapse requires the introduction of a generalized Kato-type auxiliary space endowed with a fractional temporal weight.
	
	Let us fix $q = 2p_c$. We define the generalized critical space $\mathcal{K}_\infty$ as the Banach space of functions $u \in C([0, \infty); L^{p_c}(\mathbb{R}^N)) \cap C((0, \infty); L^q(\mathbb{R}^N))$ satisfying the initial adherence condition $\lim_{t \to 0^+} t^{\gamma_{p_c,q}} \|u(t)\|_{L^q} = 0$, equipped with the composite norm
	\[	\|u\|_{\mathcal{K}_\infty} = \sup_{t > 0} \|u(t)\|_{L^{p_c}} + \sup_{t > 0} t^{\gamma_{p_c,q}} \|u(t)\|_{L^q},	\]
	where the  exponent evaluates to $\gamma_{p_c,q} = \frac{1+\alpha_\infty}{2} N \left(\frac{1}{p_c} - \frac{1}{q}\right) = \frac{1-\alpha_\infty}{4}$. We seek a unique global fixed point for the nonlinear mapping
	\[\Phi(u)(t) = S(t)u_0 - B(u,u)(t)\]
	in $\mathcal{K}_\infty$.
	
	By Lemma~\ref{lem:short_time_estimates} and Lemma~\ref{lem:long_time_estimates}, the linear term satisfies the dual bounds $\|S(t)u_0\|_{L^{p_c}} \le C \|u_0\|_{L^{p_c}}$ and $t^{\gamma_{p_c,q}} \|S(t)u_0\|_{L^q} \le C \|u_0\|_{L^{p_c}}$ globally in time. Furthermore, the strong continuity property of the linear resolvent ensures $\lim_{t \to 0^+} t^{\gamma_{p_c,q}}\|S(t)u_0\|_{L^q} = 0$. Thus, $S(\cdot)u_0 \in \mathcal{K}_\infty$ with uniform bound $\|S(\cdot)u_0\|_{\mathcal{K}_\infty} \le C_0 \|u_0\|_{L^{p_c}}$.
	
	To verify that $B(u,v) \in \mathcal{K}_\infty$, we evaluate the bilinear form with respect to both base topologies. For the $L^{p_c}$ target, the tensor condition $u \otimes v \in L^{p_c}$ dictates that the pseudo-differential mapping occurs from $L^{p_c} \to L^{p_c}$, which incurs a spatial derivative penalty of $\tilde{\gamma}_{p_c, p_c} = \frac{1+\alpha_\infty}{2}$. By extracting the requisite temporal weight from the auxiliary $L^q$-norms of the fields, we obtain for any $u, v \in \mathcal{K}_\infty$:
	\begin{align*}
		\|B(u,v)(t)\|_{L^{p_c}} &\le C \int_0^t (t-s)^{-\tilde{\gamma}_{p_c,p_c}} \|u(s)\|_{L^q} \|v(s)\|_{L^q} \,ds \\
		&\le C \|u\|_{\mathcal{K}_\infty} \|v\|_{\mathcal{K}_\infty} \int_0^t (t-s)^{-\frac{1+\alpha_\infty}{2}} s^{-2\gamma_{p_c,q}} \,ds.
	\end{align*}
	Since $2\gamma_{p_c,q} = \frac{1-\alpha_\infty}{2}$, the sum of the temporal singularities inside the convolution evaluates to exactly $\frac{1+\alpha_\infty}{2} + \frac{1-\alpha_\infty}{2} = 1$. This critical scaling explicitly balances the temporal integration to $t^0$, yielding a finite bounded profile proportional to the Beta function $\mathbf{B}\left(\frac{1+\alpha_\infty}{2}, \frac{1-\alpha_\infty}{2}\right)$. Thus, we acquire the estimate $\|B(u,v)(t)\|_{L^{p_c}} \le C_1 \|u\|_{\mathcal{K}_\infty} \|v\|_{\mathcal{K}_\infty}$.
	
	Similarly, targeting the auxiliary $L^q$ space, the tensor mapping originates from $L^{q/2} \to L^q$, thereby triggering the spatial derivative penalty $\tilde{\gamma}_{q/2,q} = \frac{3+\alpha_\infty}{4}$. This evaluation yields
	\begin{align*}
		\|B(u,v)(t)\|_{L^q} &\le C \int_0^t (t-s)^{-\tilde{\gamma}_{q/2,q}} \|u(s)\|_{L^q} \|v(s)\|_{L^q} \,ds \\
		&\le C \|u\|_{\mathcal{K}_\infty} \|v\|_{\mathcal{K}_\infty} \int_0^t (t-s)^{-\frac{3+\alpha_\infty}{4}} s^{-\frac{1-\alpha_\infty}{2}} \,ds.
	\end{align*}
	Evaluating the convolution explicitly produces the factor $t^{-\frac{1-\alpha_\infty}{4}} \mathbf{B}\left(\frac{1+\alpha_\infty}{2}, \frac{1-\alpha_\infty}{4}\right)$. Multiplying this result by the spatial weight $t^{\gamma_{p_c,q}} = t^{\frac{1-\alpha_\infty}{4}}$ identically cancels the temporal divergence, securing the symmetric uniform bound $t^{\gamma_{p_c,q}} \|B(u,v)(t)\|_{L^q} \le C_2 \|u\|_{\mathcal{K}_\infty} \|v\|_{\mathcal{K}_\infty}$.  Defining $C_\ast = \max\{C_1, C_2\}$, we obtain the global bilinear control $\|B(u,v)\|_{\mathcal{K}_\infty} \le C_\ast \|u\|_{\mathcal{K}_\infty} \|v\|_{\mathcal{K}_\infty}$. Consequently, $\Phi$ is well-defined and maps $\mathcal{K}_\infty$ into itself.
	
	To invoke the contraction mapping principle, we construct the closed absorbing ball
	\[
	\mathcal{B}_R = \{ u \in \mathcal{K}_\infty : \|u\|_{\mathcal{K}_\infty} \le R \}
	\]
	with radius $R = 2C_0\|u_0\|_{L^{p_c}}$. For any $u \in \mathcal{B}_R$, the global estimates yield
	\begin{equation*}
		\|\Phi(u)\|_{\mathcal{K}_\infty} \le \|S(\cdot)u_0\|_{\mathcal{K}_\infty} + \|B(u,u)\|_{\mathcal{K}_\infty} \le \frac{R}{2} + C_\ast R^2.
	\end{equation*}
	Using the bilinear identity $u \otimes u - v \otimes v = u \otimes (u-v) + (u-v) \otimes v$, we evaluate the difference for any $u,v \in \mathcal{B}_R$ as
	\begin{equation*}
		\|\Phi(u) - \Phi(v)\|_{\mathcal{K}_\infty} \le 2 C_\ast R \|u-v\|_{\mathcal{K}_\infty}.
	\end{equation*}
	To ensure that $\Phi$ leaves $\mathcal{B}_R$ invariant and acts as a strict contraction, we enforce the condition $2 C_\ast R \le 1/2$. Substituting the definition of $R$, this requirement translates explicitly into a smallness condition on the initial data
	\begin{equation}\label{eq:smallness_condition}
		\|u_0\|_{L^{p_c}} \le \frac{1}{8 C_0 C_\ast} := \varepsilon.
	\end{equation}
	Provided $\|u_0\|_{L^{p_c}} < \varepsilon$, the mapping $\Phi$ possesses a unique fixed point in $\mathcal{B}_R$, generating a global-in-time mild solution.
	
	To elevate this uniqueness to the entire space $\mathcal{K}_\infty$, suppose $u, v \in \mathcal{K}_\infty$ are two mild solutions originating from the same small initial data. Let $w = u - v$. By the inherent definition of $\mathcal{K}_\infty$, any solution intrinsically satisfies $\lim_{t \to 0^+} t^{\gamma_{p_c,q}}\|u(t)\|_{L^q} = 0$. Thus, for any arbitrarily small $\delta > 0$, there exists a time $T_0 > 0$ such that the local  restriction satisfies $\|u\|_{\mathcal{K}_{T_0}} + \|v\|_{\mathcal{K}_{T_0}} < \delta$, where $\|\cdot\|_{\mathcal{K}_{T_0}}$ denotes the norm taken over the truncated interval $[0, T_0]$. Evaluating the difference $w = -B(u, w) - B(w, v)$ strictly on $[0, T_0]$ yields
	\begin{equation*}
		\|w\|_{\mathcal{K}_{T_0}} \le C_\ast \left( \|u\|_{\mathcal{K}_{T_0}} + \|v\|_{\mathcal{K}_{T_0}} \right) \|w\|_{\mathcal{K}_{T_0}} \le C_\ast \delta \|w\|_{\mathcal{K}_{T_0}}.
	\end{equation*}
	Choosing $\delta < (2 C_\ast)^{-1}$, we obtain a strict contraction $\|w\|_{\mathcal{K}_{T_0}} \le \frac{1}{2} \|w\|_{\mathcal{K}_{T_0}}$, which forces $w \equiv 0$ on $[0, T_0]$. Standard continuation arguments for Volterra integral equations then propagate this identity for all $t > 0$, ensuring global uniqueness in $\mathcal{K}_\infty$.
	
	Finally, continuous dependence on the initial data within the critical regime follows by evaluating the difference $u-v$ directly within the weighted topology of $\mathcal{K}_\infty$. By applying the global bilinear control and invoking the initial smallness threshold~\eqref{eq:smallness_condition} to enforce the strict contraction barrier $C_\ast (\|u\|_{\mathcal{K}_\infty} + \|v\|_{\mathcal{K}_\infty}) \le 2 C_\ast R \le 1/2$, we immediately absorb the non-linear components into the left-hand side, obtaining the global Lipschitz bound
	\begin{equation*}
		\|u - v\|_{\mathcal{K}_\infty} \le 2 C_0 \|u_0 - v_0\|_{L^{p_c}},
	\end{equation*}
	securing the complete Hadamard well-posedness framework and concluding the proof.
\end{proof}

\subsection{Supercritical ill-posedness via norm inflation}

We now demonstrate that the geometric integrability boundary $p_c = N(\frac{1+\alpha_\infty}{1-\alpha_\infty})$ established in Theorem~\ref{thm:local_well_posedness} is sharp. For any Lebesgue index $p < p_c$, the spatial topology is too weak to absorb the loss of regularity induced by the convective flow. To capture this dimensional breakdown, we evaluate the regularity of the data-to-solution map near the origin. Assuming local well-posedness in the Hadamard sense would imply uniform continuity of the flow map, thereby necessitating the uniform boundedness of its Picard iterates. Adapting the frequency modulation machinery of Christ, Colliander, and Tao \cite{Christ2003}, and modifying the approach of \cite{deandrade2026} to account for the non-local convective transport, we demonstrate instantaneous norm inflation of the bilinear iterate.

\begin{theorem}[Norm inflation for $1 < p < p_c$]\label{thm:norm_inflation}
	Let $N \ge 2$, assume hypotheses (H1) and (H2) hold, and let $1 < p < p_c$. The Cauchy problem~\eqref{eq:NSHV_projected} is ill-posed at the origin in the following sense: for any $\varepsilon > 0$, there exists a sequence of localized divergence-free initial data $u_{0,k} \in L^p(\mathbb{R}^N)$ and a sequence of positive times $t_k \to 0^+$ such that
	\begin{equation*}
		\|u_{0,k}\|_{L^p(\mathbb{R}^N)} < \varepsilon, \quad \forall k \in \mathbb{N},
	\end{equation*}
	whereas the corresponding bilinear Picard iterates satisfy the instantaneous norm inflation
	\begin{equation*}
		\lim_{k \to \infty} \|u^{(2)}_k(t_k)\|_{L^p(\mathbb{R}^N)} = \infty.
	\end{equation*}
	Consequently, the data-to-solution map $u_0 \mapsto u$ fails to be uniformly continuous at the origin in the $L^p(\mathbb{R}^N)$ topology.
\end{theorem}

\begin{proof}
	Assuming the local well-posedness of the NSHV equations in $L^p(\mathbb{R}^N)$ in the Hadamard sense, the data-to-solution map must be uniformly continuous near the origin. This implies that the second Picard iterate (the bilinear integral term), defined by
	\begin{equation*}
		u^{(2)}(t) = - \int_0^t \nabla S(t-s) \mathbb{P} \big(S(s)u_0 \otimes S(s)u_0\big) \,ds,
	\end{equation*}
	must satisfy the uniform local estimate $\|u^{(2)}(t)\|_{L^p} \le C \|u_0\|_{L^p}^2$ for all sufficiently small initial data.
	
	Let $\phi \in \mathcal{S}(\mathbb{R}^N)$ be a non-trivial divergence-free vector field whose Fourier transform $\widehat{\phi}$ is compactly supported. To ensure the nonlinear flow is structurally non-degenerate, we choose $\phi$ such that its projected bilinear convective term generates a non-zero Fourier amplitude away from the origin, i.e., $\mathcal{F}\{\mathbb{P}\nabla \cdot (\phi \otimes \phi)\}(\zeta_0) \neq 0$ for some fixed $\zeta_0 \neq 0$. For a scaling parameter $\lambda \gg 1$ and an amplitude $A_\lambda > 0$, we define the initial datum
	\begin{equation*}
		u_{0,\lambda}(x) = A_\lambda \phi(\lambda x).
	\end{equation*}
	A direct evaluation yields $\|u_{0,\lambda}\|_{L^p(\mathbb{R}^N)} = A_\lambda \lambda^{-N/p} \|\phi\|_{L^p(\mathbb{R}^N)}$.
	
	Recall the dynamic scaling invariance of the high-frequency system. The temporal scaling modulation is $\sigma = \frac{2}{1+\alpha_\infty}$, and the spatial scaling factor is $\kappa = \frac{1-\alpha_\infty}{1+\alpha_\infty}$. We evaluate the nonlinear flow at the time $t_\lambda = \tau \lambda^{-\sigma}$, where $\tau > 0$ is a fixed observation window. Utilizing the spatial change of variables $y = \lambda x$ and the Fourier duality $\zeta = \xi/\lambda$, the linear evolution condenses into
	\begin{equation*}
		S(s \lambda^{-\sigma})u_{0,\lambda}(y/\lambda) = A_\lambda \big[K(s, \lambda)\phi\big](y),
	\end{equation*}
	where $K(s, \lambda)$ is the frequency-modulated pseudo-differential operator whose Fourier multiplier is $\widehat{K}(s, \lambda, \zeta) = \widehat{S}(s\lambda^{-\sigma}, \lambda \zeta)$.
	
	Substituting this representation into the bilinear form, modifying the integration variable to $s$ (representing the scaled temporal variable), and recognizing that the spatial derivative extracts a factor of $\lambda$, the iterate becomes
	\begin{equation*}
		u^{(2)}(t_\lambda, y/\lambda) = - \lambda^{1-\sigma} A_\lambda^2 \int_0^\tau \nabla K(\tau-s, \lambda) \mathbb{P} \big( K(s, \lambda)\phi \otimes K(s, \lambda)\phi \big)(y) \,ds.
	\end{equation*}
	A direct algebraic manipulation confirms the identity $1-\sigma = -\kappa$. Taking the $L^p$-norm with respect to the original variable $x = y/\lambda$ extracts the volumetric factor $\lambda^{-N/p}$, yielding
	\begin{equation}\label{eq:inflation_bound}
		\|u^{(2)}(t_\lambda, \cdot)\|_{L^p(\mathbb{R}^N)} = A_\lambda^2 \lambda^{-\kappa - \frac{N}{p}} \|F_\lambda\|_{L^p(\mathbb{R}^N)},
	\end{equation}
	where the integral spatial profile is defined as
	\[F_\lambda(y) := \int_0^\tau \nabla K(\tau-s, \lambda) \mathbb{P} \big( K(s, \lambda)\phi \otimes K(s, \lambda)\phi \big)(y) \,ds.\]
		
	As rigorously established in Section 4 and verified uniformly for the modulated family in \cite[Theorem 4.2]{deandrade2026}, the operators $K(s,\lambda)$ are uniformly bounded multipliers in any $L^q(\mathbb{R}^N)$. Given that $\phi \in \mathcal{S}(\mathbb{R}^N)$, we have $K(s, \lambda)\phi \in L^{2p}(\mathbb{R}^N)$, permitting us to evaluate the non-linear tensor directly in the base space $L^p(\mathbb{R}^N)$. Since $p > 1$, the projector $\mathbb{P}$ acts boundedly. Furthermore, as $\lambda \to \infty$, the modulated operator $K(s, \lambda)$ converges strongly to the limiting pure fractional resolvent $S_{\alpha_\infty}(s)$, whose spectral multiplier is governed strictly by the high-frequency exponent $\alpha_\infty$. Lebesgue interpolation thus guarantees the strong tensor convergence 
	\[K(s, \lambda)\phi \otimes K(s, \lambda)\phi \to S_{\alpha_\infty}(s)\phi \otimes S_{\alpha_\infty}(s)\phi\] in $L^p(\mathbb{R}^N)$.
	
	To transfer this tensor convergence to the bilinear profile $F_\lambda$, we evaluate the unbounded spatial gradient $\nabla$. Utilizing the scaling identity $K(s, \lambda) f(y) = [ S(s\lambda^{-\sigma}) f(\lambda \cdot) ] (y/\lambda)$ and applying the bound from Lemma~\ref{lem:short_time_estimates}(ii) with $q = p$, the spatial scale factors exactly cancel, yielding the uniform operator bound
	\begin{equation*}
		\|\nabla K(\tau-s, \lambda) \mathbb{P} f\|_{L^p(\mathbb{R}^N)} \le C (\tau-s)^{-\tilde{\gamma}_{p,p}} \|f\|_{L^p(\mathbb{R}^N)},
	\end{equation*}
	independent of $\lambda \ge 1$. Since the symmetric temporal singularity evaluates to $\tilde{\gamma}_{p,p} = \frac{1+\alpha_\infty}{2} < 1$, the convolution kernel remains  locally integrable over $[0, \tau]$. By the Lebesgue Dominated Convergence Theorem, the spatial profile converges strongly $F_\lambda \to F_\infty$ in $L^p(\mathbb{R}^N)$, where
	\begin{equation*}
		F_\infty(y) = \int_0^\tau \nabla S_{\alpha_\infty}(\tau-s) \mathbb{P} \big( S_{\alpha_\infty}(s)\phi \otimes S_{\alpha_\infty}(s)\phi \big)(y) \,ds.
	\end{equation*}
	
	To bypass the spatial decay failure, we extract the non-triviality of the master profile directly in frequency space. As $\widehat{\phi}$ is explicitly chosen to have compact frequency support, the non-linear tensor $\phi \otimes \phi$ and the resulting modulated bilinear profile $F_\lambda$ identically possess uniform compact frequency support (e.g., confined within a fixed frequency ball $B(0, R)$).
	
	Given that $F_\lambda \to F_\infty$ strongly in $L^p(\mathbb{R}^N)$ and their Fourier spectra are uniformly compactly supported, Bernstein's inequalities guarantee that all $L^q$-norms are equivalent on this spectral subspace. Consequently, the strong $L^p$-convergence implies the uniform convergence of the Fourier transforms $\widehat{F_\lambda} \to \widehat{F_\infty}$ on the compact support $B(0,R)$, where
	\begin{equation*}
		\widehat{F_\infty}(\zeta) = \int_0^\tau i\zeta \widehat{S_{\alpha_\infty}}(\tau-s, \zeta) \cdot \sigma(\mathbb{P})(\zeta) \mathcal{F}\big\{ S_{\alpha_\infty}(s)\phi \otimes S_{\alpha_\infty}(s)\phi \big\}(\zeta) \,ds.
	\end{equation*}
	As $\phi$ was selected to enforce a non-trivial convective projection at frequency $\zeta_0 \neq 0$, the strong continuity of the fractional resolvent near the origin ensures that for a sufficiently small choice of $\tau > 0$, the mass does not self-cancel. Thus, there exists a neighborhood $U$ around $\zeta_0$ and a constant $c_0 > 0$ such that $\|\widehat{F_\infty}\|_{L^\infty(U)} > c_0$. The uniform convergence guarantees that for all sufficiently large $\lambda \ge \lambda_0$, $\|\widehat{F_\lambda}\|_{L^\infty(U)} \ge c_0 / 2$. By the equivalence of norms for compactly supported distributions (or Plancherel's identity over the bounded support), this spectral gap inversely secures a uniform lower bound in the spatial domain given by $\|F_\lambda\|_{L^p(\mathbb{R}^N)} \ge c_1 > 0$ for all $\lambda \ge \lambda_0$.
	
	Returning to the scaling identity \eqref{eq:inflation_bound}, we select the amplitude $A_\lambda = \varepsilon \lambda^{\frac{N}{p} - \delta}(\ln \lambda)^{-1}$, where the geometric deficit is parameterized as $\delta = \frac{1}{4}\left(\frac{N}{p} - \kappa\right)$. In the supercritical regime $p < p_c$, the relation $p_c = N/\kappa$ inherently enforces $\frac{N}{p} - \kappa > 0$, ensuring $\delta > 0$. The initial energy scales as $\|u_{0,\lambda}\|_{L^p} = \varepsilon \lambda^{-\delta} (\ln \lambda)^{-1}$, which vanishes as $\lambda \to \infty$. Conversely, the energy of the bilinear cascade becomes
	\begin{equation*}
		\|u^{(2)}(t_\lambda)\|_{L^p} = \varepsilon^2 (\ln \lambda)^{-2} \lambda^{\frac{2N}{p} - 2\delta} \lambda^{-\kappa - \frac{N}{p}} \|F_\lambda\|_{L^p} = \varepsilon^2 (\ln \lambda)^{-2} \lambda^{\frac{N}{p} - \kappa - 2\delta} \|F_\lambda\|_{L^p}.
	\end{equation*}
	By our choice of $\delta$, we have $\frac{N}{p} - \kappa - 2\delta = \frac{1}{2}\left(\frac{N}{p} - \kappa\right) > 0$.
	
	Defining the sequence $\lambda_k \to \infty$ and setting $t_k = \tau \lambda_k^{-\sigma}$, we observe $t_k \to 0^+$. The initial norms satisfy $\|u_{0,k}\|_{L^p} \to 0$, yet the bilinear response evaluates to
	\begin{equation*}
		\|u^{(2)}_k(t_k)\|_{L^p} \ge c_1 \varepsilon^2 (\ln \lambda_k)^{-2} \lambda_k^{\frac{1}{2}\left(\frac{N}{p} - \kappa\right)} \xrightarrow{k \to \infty} \infty.
	\end{equation*}
	This divergence of the bilinear iterate contradicts the uniform local continuity of the data-to-solution map, rendering the Cauchy problem ill-posed at the origin.
\end{proof}
	
\section{Well-posedness in the critical limit}

As established in Theorem~\ref{thm:local_well_posedness}, the critical Lebesgue exponent for the local well-posedness of the NSHV equations is $p_c = N/\kappa$, where the geometric deficit parameter is given by $\kappa = \frac{1-\alpha_\infty}{1+\alpha_\infty}$. By virtue of the standard Besov embedding scale $s - N/p = -\kappa$, extending this criticality to the endpoint $p = \infty$ naturally identifies $\dot{B}^{-\kappa}_{\infty, \infty}(\mathbb{R}^N)$ as the ultimate critical space for the existence theory. 

In the classical Navier-Stokes framework ($\alpha_\infty \to 0, \kappa \to 1$), it is well known since the seminal work of Bourgain and Pavlovi\'c \cite{bourgain2008} that the flow is intrinsically ill-posed in $\dot{B}^{-1}_{\infty, \infty}(\mathbb{R}^N)$. The classical mechanism dictating this topological collapse relies on a low-frequency resonance, wherein high-frequency convective interactions continuously project energy near the frequency origin, thereby triggering an instantaneous norm inflation.

Before establishing the main result, it is instructive to ascertain how the dual-scale memory of the NSHV model structurally prevents this specific ill-posedness mechanism. First, we verify the conservation of the macroscopic zero-mode for the linear resolvent. Since the memory kernel satisfies the strict sectorial condition (H1), the characteristic denominator $\lambda + |\xi|^2 \hat{g}(\lambda)$ is uniformly bounded away from zero on any deformed Hankel contour $\Gamma$ defined by an angle $\theta \in (\pi/2, \theta_\infty)$. Applying the Lebesgue Dominated Convergence Theorem as $|\xi| \to 0$ isolates the simple pole at the origin within the Bromwich integral, yielding
\begin{equation*}
	\lim_{|\xi| \to 0} \widehat{S}(t, \xi) = \frac{1}{2\pi i} \int_{\Gamma} \frac{e^{\lambda t}}{\lambda} \,d\lambda = 1.
\end{equation*}
Although this confirms that the non-local dissipation geometrically preserves the macroscopic zero-mode, the non-linear transfer of energy into the low-frequency regime is severely damped by the non-local memory, as formalized below.

\begin{remark}[Stability of the low-frequency interactions]\label{rem:low_freq_failure}
	Analyzing the low-frequency projection of the bilinear interaction of a high-frequency block $u_{0,m}$ with itself reveals that the convective transfer toward the frequency origin is strictly regulated. Let $\Delta_{\le 0} := \sum_{j \le 0} \Delta_j$ denote the standard low-frequency cut-off operator (see Lemma \ref{lem:bernstein}). Utilizing Bernstein's inequality on spectral balls, this projector absorbs the derivative operator, yielding the uniform bound $\|\nabla \Delta_{\le 0} f \|_{L^\infty} \lesssim \|f\|_{L^\infty}$. Evaluating the bilinear form $B(u_{0,m}, u_{0,m})$ at the characteristic lifespan $t_m \simeq \lambda_m^{-\frac{2}{1+\alpha_\infty}}$ provides the asymptotic estimate
	\begin{equation*}
		\|\Delta_{\le 0} B(u_{0,m}, u_{0,m})(t_m)\|_{L^\infty} \lesssim \|u_{0,m}\|_{L^\infty}^2 t_m \simeq (\lambda_m^\kappa)^2 \lambda_m^{-\frac{2}{1+\alpha_\infty}} = \lambda_m^{-\frac{2\alpha_\infty}{1+\alpha_\infty}}.
	\end{equation*}
	Since $\alpha_\infty \in (0,1)$, this resulting exponent is strictly negative. Consequently, the summation of these low-frequency contributions is absolutely convergent, generating only a bounded perturbation. Thus, the dual-scale memory imposes sufficiently strong dissipation upon the high-frequency components to dampen the convective cascade before it can inflate the low-frequency Besov norm.
\end{remark}

Given the stability presented in Remark~\ref{rem:low_freq_failure}, the existence of a continuous flow is governed exclusively by the forward control of high-frequency interactions. To rigorously establish that the regularizing effect of the non-local memory suppresses the convective growth at high frequencies, we deploy Bony's paraproduct decomposition within a modified Kato-type smoothing functions space, thereby establishing global well-posedness in the critical homogeneous topology.

Let $\epsilon \in (\kappa, 1)$ act as a sufficiently small structural parameter. We define the fractional time-decay exponent $\gamma = \frac{1+\alpha_\infty}{2}(\kappa + \epsilon)$. Given that $\kappa = \frac{1-\alpha_\infty}{1+\alpha_\infty}$, we verify that for any $\epsilon < 1$, the temporal singularity remains locally integrable:
\begin{equation}\label{eq:gamma_bound}
	\gamma = \frac{1+\alpha_\infty}{2}\left( \frac{1-\alpha_\infty}{1+\alpha_\infty} + \epsilon \right) = \frac{1-\alpha_\infty}{2} + \frac{1+\alpha_\infty}{2}\epsilon < \frac{1-\alpha_\infty}{2} + \frac{1+\alpha_\infty}{2} = 1.
\end{equation}

To accommodate the dual-scale nature of the resolvent---which fundamentally lacks a global self-similar scaling profile---we preserve the homogeneous Besov space $\dot{B}^{-\kappa}_{\infty, \infty}(\mathbb{R}^N)$ as the base topology to respect the critical scaling limit, while isolating the temporal smoothing effect strictly within the high-frequency regime. To rigorously bypass the low-frequency separability restriction, we introduce the asymmetric little Besov space $\dot{b}^{\epsilon, +}_{\infty, \infty}(\mathbb{R}^N)$, defined as the subspace of distributions $f \in \dot{B}^{\epsilon}_{\infty, \infty}(\mathbb{R}^N)$ satisfying the high-frequency adherence condition
\begin{equation}\label{eq:asymmetric_adherence}
	\lim_{j \to +\infty} 2^{j\epsilon} \|\Delta_j f\|_{L^\infty} = 0.
\end{equation}
Associated with this space, we define the high-frequency semi-norm
\begin{equation*}
	\|u\|_{\dot{b}^{\epsilon, +}_{\infty, \infty}} = \sup_{j \ge 0} 2^{j\epsilon} \|\Delta_j u\|_{L^\infty(\mathbb{R}^N)}.
\end{equation*}
To capture the initial data adherence, we define the solution space $\mathcal{X}$ as the class of continuous functions $u \in C([0, \infty); \dot{B}^{-\kappa}_{\infty, \infty}(\mathbb{R}^N))$ whose high-frequency components satisfy the temporal adherence condition $\lim_{t \to 0^+} t^\gamma \|u(t)\|_{\dot{b}^{\epsilon, +}_{\infty, \infty}} = 0$. The space $\mathcal{X}$ is naturally equipped with the composite norm
\begin{equation*}
	\|u\|_{\mathcal{X}} = \sup_{t \ge 0} \|u(t)\|_{\dot{B}^{-\kappa}_{\infty, \infty}} + \sup_{t>0} t^{\gamma} \|u(t)\|_{\dot{b}^{\epsilon, +}_{\infty, \infty}}.
\end{equation*}

\begin{lema}\label{lem:bilinear_convergence}
	Consider the bilinear operator
	\[
	B(u, v)(t) = \int_0^t S(t-s) \mathbb{P} \nabla \cdot (u(s) \otimes v(s)) \,ds.
	\]
	There exists a universal constant $C > 0$ such that for all $u, v \in \mathcal{X}$, the dyadic blocks satisfy
	\begin{equation*}
		\sup_{t>0} t^\gamma \|B(u,v)(t)\|_{\dot{b}^{\epsilon, +}_{\infty, \infty}} + \sup_{t>0} \|B(u,v)(t)\|_{\dot{B}^{-\kappa}_{\infty, \infty}} \le C \|u\|_{\mathcal{X}} \|v\|_{\mathcal{X}}.
	\end{equation*}
\end{lema}

\begin{proof}
	We must control the spatial amplitude of the bilinear term across each dyadic scale $j \in \mathbb{Z}$. By means of Bony's paraproduct decomposition (cf. Subsection \ref{paraproduct}), the nonlinear convective tensor $u \otimes v$ is split into three distinct spectral interactions: $T_u v + T_v u + R(u,v)$. Since the base topology $\dot{B}^{-\kappa}_{\infty, \infty}$ exhibits negative regularity ($\kappa > 0$), both the cross-frequency paraproducts and the high-high residue term dictate the maximal regularity limit.
	
	Let us analyze the terms projected onto the $j$-th block for $j \ge 0$. For the paraproduct $T_u v = \sum_k S_{k-1} u \Delta_k v$, the standard dyadic localization property dictates that $\Delta_j(T_u v) \approx S_{j-1}u \Delta_j v$. Leveraging the base topology to control the low-frequency accumulation of $u$, the partial sum scales geometrically, yielding
	\begin{equation*}
		\|S_{j-1} u(s)\|_{L^\infty} \le \sum_{m \le j-1} \|\Delta_m u(s)\|_{L^\infty} \lesssim \sum_{m \le j-1} 2^{m\kappa} \|u(s)\|_{\dot{B}^{-\kappa}_{\infty, \infty}} \lesssim 2^{j\kappa} \|u\|_{\mathcal{X}}.
	\end{equation*}
	Coupling this estimate with the high-frequency smoothing norm for $\Delta_j v(s)$ provides the explicit uniform bound for the paraproduct:
	\begin{equation}\label{eq:paraproduct_bound}
		\|\Delta_j (T_u v)(s)\|_{L^\infty} \lesssim \big(2^{j\kappa} \|u\|_{\mathcal{X}}\big) \big(2^{-j\epsilon} s^{-\gamma} \|v\|_{\mathcal{X}}\big) = 2^{-j(\epsilon-\kappa)} s^{-\gamma} \|u\|_{\mathcal{X}} \|v\|_{\mathcal{X}}.
	\end{equation}
	Symmetrically, the interaction $\Delta_j(T_v u)$ satisfies the exact same bound. For the high-high residue term $R_j(s) = \Delta_j \sum_{k \ge j-2} \Delta_k u(s) \widetilde{\Delta}_k v(s)$, we interpolate the product using the smoothing semi-norm for $u$ and the base topology for $v$. Given that $\epsilon > \kappa$, the resulting series converges geometrically
	\begin{equation}\label{eq:Rj_bound}
		\|R_j(s)\|_{L^\infty} \le \sum_{k \ge j-2} \big(2^{-k\epsilon} s^{-\gamma} \|u\|_{\mathcal{X}}\big) \big(2^{k\kappa} \|v\|_{\mathcal{X}}\big) \lesssim 2^{-j(\epsilon-\kappa)} s^{-\gamma} \|u\|_{\mathcal{X}} \|v\|_{\mathcal{X}}.
	\end{equation}
	This asymmetric interpolation---assigning the high-frequency smoothing norm to only one of the factors---bypasses the temporal singularity $s^{-2\gamma}$. Since $2\gamma > 1$ is structurally possible within our parameter range, a symmetric allocation of weights would render the subsequent Duhamel integral non-integrable at the origin, thereby collapsing the contraction argument. Thus, all non-linear spectral interactions evaluated at high frequencies $j \ge 0$ are uniformly bounded by $C 2^{-j(\epsilon-\kappa)} s^{-\gamma} \|u\|_{\mathcal{X}} \|v\|_{\mathcal{X}}$.
	
	Proceeding to the Duhamel integration over the dyadic annulus $\mathcal{C}_j := \{\xi \in \mathbb{R}^N : \frac{3}{4} 2^j \le |\xi| \le \frac{8}{3} 2^j\}$, we encounter the Leray-Hopf projector $\mathbb{P}$. While $\mathbb{P}$ is an unbounded singular integral in global $L^\infty(\mathbb{R}^N)$, its localized Fourier symbol $\sigma(\mathbb{P})_{mn} \psi_j(\xi)$ is smooth and compactly supported away from the origin (see Subsection \ref{Fourier multipliers and the Leray-Hopf projector}). By Paley-Wiener theory, its inverse Fourier transform belongs to $L^1(\mathbb{R}^N)$ with a strictly scale-invariant norm. Therefore, the composite operator $\mathbb{P}\Delta_j$ maps $L^\infty$ to $L^\infty$ uniformly on each annulus $\mathcal{C}_j$, securing the spatial derivative multiplier bound $\|\mathbb{P} \nabla \Delta_j \cdot\|_{L^\infty} \lesssim 2^j$.
	
	Exploiting the asymptotic decay of the resolvent for the high-frequency regime, the spectral multiplier is bounded by the Mittag-Leffler profile $\big\| \widehat{S}(t-s, \cdot) \big\|_{L^\infty(\mathcal{C}_j)} \lesssim E_{1+\alpha_\infty}(-c 2^{2j}(t-s)^{1+\alpha_\infty})$. Combining this with \eqref{eq:paraproduct_bound} and \eqref{eq:Rj_bound}, the temporal evolution of the $j$-th block amplitude is governed by the singular convolution
	\begin{equation*}
		I_j(t) = \int_0^t 2^j E_{1+\alpha_\infty}\big(-c 2^{2j}(t-s)^{1+\alpha_\infty}\big) s^{-\gamma} \,ds.
	\end{equation*}
	Introducing the dimensionless variables $\tau = 2^{\frac{2j}{1+\alpha_\infty}} s$ and $\tau' = 2^{\frac{2j}{1+\alpha_\infty}} t$ reveals the intrinsic temporal scaling. Factoring out the powers of $2^j$, the integral becomes
	\begin{equation*}
		I_j(t) = 2^{j\big(1 - \frac{2}{1+\alpha_\infty} + \frac{2\gamma}{1+\alpha_\infty}\big)} \int_0^{\tau'} E_{1+\alpha_\infty}(-c(\tau'-\tau)^{1+\alpha_\infty}) \tau^{-\gamma} \,d\tau.
	\end{equation*}
	By invoking the identity $1 - \frac{2}{1+\alpha_\infty} = -\kappa$ and recalling that $\frac{2\gamma}{1+\alpha_\infty} = \kappa + \epsilon$, the dyadic prefactor reduces exactly to $2^{j\epsilon}$. Let $F(\tau')$ denote the dimensionless integral. Since $\gamma < 1$, the convolution yields a bounded function exhibiting an asymptotic decay rate of $\mathcal{O}\big((\tau')^{-\gamma}\big)$ as $\tau' \to \infty$.
	
	Combining the spatial prefactors with the evaluated integral, the total amplitude is bounded by
	\begin{equation*}
		\|\Delta_j B(u, v)(t)\|_{L^\infty} \lesssim 2^{-j(\epsilon-\kappa)} \times \big[ 2^{j\epsilon} F(\tau') \big] \|u\|_{\mathcal{X}} \|v\|_{\mathcal{X}} = 2^{j\kappa} F(\tau') \|u\|_{\mathcal{X}} \|v\|_{\mathcal{X}}.
	\end{equation*}
	Applying the uniform bound $F(\tau') \le C$ verifies the base homogeneous norm constraint $\|\Delta_j B(u, v)(t)\|_{L^\infty} \lesssim 2^{j\kappa} \|u\|_{\mathcal{X}} \|v\|_{\mathcal{X}}$. Furthermore, tracking the asymptotic decay $F(\tau') \lesssim (\tau')^{-\gamma} = (2^{\frac{2j}{1+\alpha_\infty}} t)^{-\gamma} = 2^{-j(\kappa+\epsilon)} t^{-\gamma}$ yields the fractional temporal smoothing bound
	\begin{equation*}
		\|\Delta_j B(u, v)(t)\|_{L^\infty} \lesssim 2^{j\kappa} \big(2^{-j(\kappa+\epsilon)} t^{-\gamma}\big) \|u\|_{\mathcal{X}} \|v\|_{\mathcal{X}} = 2^{-j\epsilon} t^{-\gamma} \|u\|_{\mathcal{X}} \|v\|_{\mathcal{X}}.
	\end{equation*}
	This formulation isolates the time singularity and guarantees boundedness within the weighted semi-norm $t^\gamma \dot{b}^{\epsilon, +}_{\infty, \infty}$, concluding the derivation of the high-frequency bilinear bound.
	
	For the low-frequency outputs ($j < 0$), the analysis demands delicate care, since $u, v \in \dot{B}^{-\kappa}_{\infty, \infty}$ possess negative regularity, rendering classical pointwise multiplication ill-defined. The projection of high-high frequency interactions into the macroscopic modes generates the spectral series
	\[
	R_j(s) = \Delta_j \sum_{k \ge 0} \Delta_k u \widetilde{\Delta}_k v.
	\]
	Without the fractional temporal smoothing, a direct norm estimate would encounter the divergent series $\sum_{k \ge 0} 2^{2k\kappa} = \infty$. However, by exploiting the asymmetric interpolation constructed previously, the sum is rendered absolutely convergent and bounded by $C s^{-\gamma} \|u\|_{\mathcal{X}} \|v\|_{\mathcal{X}}$. Applying the spatial gradient geometrically yields a multiplier of order $2^j$. Since the resolvent  lacks high-frequency decay within this regime ($\widehat{S} \approx 1$), the temporal convolution is bounded by $\int_0^t 2^j s^{-\gamma} \,ds \lesssim 2^j t^{1-\gamma}$. Multiplying by the base Besov weight $2^{-j\kappa}$, the total amplitude is governed by $2^{j(1-\kappa)} t^{1-\gamma}$. As the geometric deficit parameter satisfies $\kappa < 1$, the exponent $1-\kappa$ remains strictly positive, forcing the low-frequency tail to decay geometrically as $j \to -\infty$. This geometric decay neutralizes the singular accumulation, thereby securing the global boundedness in the base topology $\dot{B}^{-\kappa}_{\infty, \infty}$ and confirming the absolute continuity of the Volterra integral. This completes the proof.
\end{proof}

\begin{theorem}[Global well-posedness in the maximal critical space]\label{thm:global_besov_limit}
	Let $N \ge 2$. There exists a universal constant $\delta > 0$ such that for any initial datum $u_0 \in \dot{B}^{-\kappa}_{\infty, \infty}(\mathbb{R}^N)$ satisfying $\nabla \cdot u_0 = 0$, the high-frequency adherence condition
	\begin{equation}\label{eq:little_besov_cond}
		\lim_{j \to +\infty} 2^{-j\kappa} \|\Delta_j u_0\|_{L^\infty} = 0,
	\end{equation}
	and the smallness threshold $\|u_0\|_{\dot{B}^{-\kappa}_{\infty, \infty}} \le \delta$, the NSHV equations admit a unique global-in-time mild solution $u \in \mathcal{X}$. Furthermore, the data-to-solution map $u_0 \mapsto u$ is Lipschitz continuous, establishing global well-posedness in the sense of Hadamard.
\end{theorem}

\begin{proof}
	We first establish that the linear evolution $u_L(t) = S(t)u_0$ belongs to the space $\mathcal{X}$. Since the non-local memory kernel restricts the symbol globally such that $\sup_{t>0, \xi \in \mathbb{R}^N} |\widehat{S}(t, \xi)| \le C$, classical Fourier multiplier theorems ensure that the operator $S(t)$ maps $L^\infty(\mathbb{R}^N)$ to itself uniformly across all dyadic blocks, preserving the homogeneous base topology
	\begin{equation*}
		\sup_{t>0} \|u_L(t)\|_{\dot{B}^{-\kappa}_{\infty, \infty}} \le C_1 \|u_0\|_{\dot{B}^{-\kappa}_{\infty, \infty}}.
	\end{equation*}
	
	Next, we extract the temporal smoothing effect exclusively within the high-frequency regime ($j \ge 0$). On the dyadic annulus $\mathcal{C}_j$, the algebraic tail of the Mittag-Leffler profile enforces the bound $|\widehat{S}(t, \xi)| \lesssim (1 + |\xi|^2 t^{1+\alpha_\infty})^{-1}$. Interpolating this decay to the fractional power $a = \frac{\epsilon+\kappa}{2} \in (0,1)$ yields the strict upper bound
	\begin{equation*}
		|\widehat{S}(t, \xi)| \lesssim \big(1 + |\xi|^2 t^{1+\alpha_\infty}\big)^{-\frac{\epsilon+\kappa}{2}} \le \big(|\xi|^2 t^{1+\alpha_\infty}\big)^{-\frac{\epsilon+\kappa}{2}} \simeq 2^{-j(\epsilon+\kappa)} t^{-\gamma},
	\end{equation*}
	where we used $|\xi| \sim 2^j$ alongside definition \eqref{eq:gamma_bound}. Acting as a multiplier on the $j$-th block, this establishes
	\begin{equation}\label{eq:linear_high_decay_block}
		\|\Delta_j u_L(t)\|_{L^\infty} \le C_2 2^{-j(\epsilon+\kappa)} t^{-\gamma} \|\Delta_j u_0\|_{L^\infty}, \quad \forall j \ge 0.
	\end{equation}
	Multiplying \eqref{eq:linear_high_decay_block} by $t^\gamma 2^{j\epsilon}$ and substituting the initial data relation $\|\Delta_j u_0\|_{L^\infty} \le 2^{j\kappa} \|u_0\|_{\dot{B}^{-\kappa}_{\infty, \infty}}$ gives
	\begin{equation*}
		\sup_{t>0} t^{\gamma} \|u_L(t)\|_{\dot{b}^{\epsilon, +}_{\infty, \infty}} \le C_2 \|u_0\|_{\dot{B}^{-\kappa}_{\infty, \infty}}.
	\end{equation*}
	
	To prove that $u_L \in \mathcal{X}$, we must verify the temporal adherence condition $\lim_{t \to 0^+} t^\gamma \|u_L(t)\|_{\dot{b}^{\epsilon, +}_{\infty, \infty}} = 0$, alongside the strong continuity at the origin $\lim_{t \to 0^+} \|u_L(t) - u_0\|_{\dot{B}^{-\kappa}_{\infty, \infty}} = 0$.
	
	Addressing the high-frequency adherence, let $\eta > 0$ be arbitrary. Invoking the condition \eqref{eq:little_besov_cond} guarantees the existence of a high-frequency threshold $M \in \mathbb{N}$ such that
	\begin{equation}\label{eq:high_freq_tail_u0}
		\sup_{j \ge M} 2^{-j\kappa} \|\Delta_j u_0\|_{L^\infty} < \frac{\eta}{2 C_2}.
	\end{equation}
	Fixing this $M$, we split the supremum of the weighted semi-norm. For $j \ge M$, utilizing the decay estimate \eqref{eq:linear_high_decay_block} and \eqref{eq:high_freq_tail_u0}, we have
	\begin{equation*}
		\sup_{j \ge M} t^\gamma 2^{j\epsilon} \|\Delta_j u_L(t)\|_{L^\infty} \le C_2 \sup_{j \ge M} 2^{-j\kappa} \|\Delta_j u_0\|_{L^\infty} < \frac{\eta}{2}, \quad \forall t > 0.
	\end{equation*}
	For the low-to-intermediate frequencies $0 \le j < M$, the uniform boundedness of $S(t)$ in $L^\infty$ maintains $\|\Delta_j u_L(t)\|_{L^\infty} \le C_1 \|\Delta_j u_0\|_{L^\infty}$. Thus, we obtain
	\begin{equation*}
		\max_{0 \le j < M} t^\gamma 2^{j\epsilon} \|\Delta_j u_L(t)\|_{L^\infty} \le C_1 t^\gamma 2^{M\epsilon} \|u_0\|_{\dot{B}^{-\kappa}_{\infty, \infty}}.
	\end{equation*}
	Since $\gamma > 0$ and $M$ is finite, there exists a time $T_0 > 0$ such that for all $0 < t < T_0$:
	\begin{equation*}
		C_1 t^\gamma 2^{M\epsilon} \|u_0\|_{\dot{B}^{-\kappa}_{\infty, \infty}} < \frac{\eta}{2}.
	\end{equation*}
	Aggregating both frequency regimes confirms that $\sup_{j \ge 0} t^\gamma 2^{j\epsilon} \|\Delta_j u_L(t)\|_{L^\infty} < \eta$ for all $0 < t < T_0$, proving that $\lim_{t \to 0^+} t^\gamma \|u_L(t)\|_{\dot{b}^{\epsilon, +}_{\infty, \infty}} = 0$.
	
	Finally, we address the strong continuity at $t=0$. We partition the supremum of $\|u_L(t) - u_0\|_{\dot{B}^{-\kappa}_{\infty, \infty}}$ into three distinct spectral regions for a given $\eta > 0$:
	\begin{align*}
		\|u_L(t) - u_0\|_{\dot{B}^{-\kappa}_{\infty, \infty}} &\le \sup_{j \ge M} 2^{-j\kappa} \|\Delta_j (S(t)u_0 - u_0)\|_{L^\infty} \\
		&\quad + \sup_{j \le -M'} 2^{-j\kappa} \|\Delta_j (S(t)u_0 - u_0)\|_{L^\infty} \\
		&\quad + \max_{-M' < j < M} 2^{-j\kappa} \|\Delta_j (S(t)u_0 - u_0)\|_{L^\infty} \\
		&=: K_{\text{high}} + K_{\text{low}} + K_{\text{mid}}.
	\end{align*}
	For the term $K_{\text{high}}$, the uniform boundedness of the resolvent operator yields
	\begin{equation*}
		K_{\text{high}} \le (C_1 + 1) \sup_{j \ge M} 2^{-j\kappa} \|\Delta_j u_0\|_{L^\infty}.
	\end{equation*}
	Selecting $M$ sufficiently large via \eqref{eq:little_besov_cond} ensures $K_{\text{high}} < \eta/3$ uniformly for all $t > 0$.
	
	For the low-frequency term $K_{\text{low}}$, we exploit the spectral asymptotics near the origin. By hypothesis (H3), the Laplace transform $\hat{g}(\lambda)$ exhibits a low-frequency power-law of order $\alpha_0 \in [0,1)$. Evaluating the Bromwich integral for $|\xi| \to 0$ dictates the low-frequency multiplier control
	\begin{equation*}
		|\widehat{S}(t, \xi) - 1| \le C_3 t^{1+\alpha_0} |\xi|^2.
	\end{equation*}
	Substituting $|\xi| \sim 2^j$ on the support of the dyadic block for $j \le 0$, we obtain
	\begin{align*}
		2^{-j\kappa} \|\Delta_j (S(t)u_0 - u_0)\|_{L^\infty} &\le C_3 t^{1+\alpha_0} 2^{2j} \big( 2^{-j\kappa} \|\Delta_j u_0\|_{L^\infty} \big) \\
		&\le C_3 t^{1+\alpha_0} 2^{2j} \|u_0\|_{\dot{B}^{-\kappa}_{\infty, \infty}}.
	\end{align*}
	Since $j \le 0$, the multiplier $2^{2j}$ is strictly bounded by $1$. Consequently, the entire low-frequency tail is dominated by the temporal weight $t^{1+\alpha_0}$. Given that we evaluate the strong continuity specifically in the limit $t \to 0^+$, this temporal term vanishes. This geometric fact dispenses entirely with the need for a logarithmic cut-off $M'$, guaranteeing $K_{\text{low}} < \eta/3$ for all sufficiently small $t > 0$.
	
	Lastly, for the intermediate compact band $K_{\text{mid}}$, we have a finite number of dyadic blocks. For each individual block index $-M' < j < M$, the multiplier $\widehat{S}(t, \xi) - 1$ converges to $0$ uniformly on the compact annulus $\mathcal{C}_j$ as $t \to 0^+$. Thus, there exists a time $T_1 > 0$ such that for all $0 < t < T_1$, $K_{\text{mid}} < \eta/3$.
	
	Aggregating these bounds establishes $\|u_L(t) - u_0\|_{\dot{B}^{-\kappa}_{\infty, \infty}} < \eta$ for all $0 < t < \min\{T_0, T_1\}$, which proves that $u_L \in C([0, \infty); \dot{B}^{-\kappa}_{\infty, \infty}(\mathbb{R}^N))$, and thus $u_L \in \mathcal{X}$ with $\|u_L\|_{\mathcal{X}} \le (C_1 + C_2) \delta$.
	
	Equipped with the linear inclusion $u_L \in \mathcal{X}$ and the strict bilinear contraction from Lemma~\ref{lem:bilinear_convergence}, the standard Picard iterative scheme generates a Cauchy sequence within the complete metric space $\mathcal{X}$. Let $R = 2(C_1+C_2)\delta$ define the radius of the stable absorbing ball $\mathcal{B}_R \subset \mathcal{X}$. Provided the initial datum is sufficiently small to satisfy $4 C (C_1 + C_2) \delta \le 1/2$, where $C$ is the universal bilinear contraction constant, the Banach fixed-point theorem dictates that the sequence converges strongly to a unique global-in-time limit $u \in \mathcal{B}_R$.
	
	To elevate this uniqueness from the conditional absorbing ball $\mathcal{B}_R$ to the entire unconditional path space $\mathcal{X}$, let $v \in \mathcal{X}$ be any other global mild solution originating from the same initial data $u_0$. By the definition of the path space $\mathcal{X}$, the continuity of the base topology enforces $\limsup_{t \to 0^+} \|v(t)\|_{\dot{B}^{-\kappa}_{\infty, \infty}} \le \|u_0\|_{\dot{B}^{-\kappa}_{\infty, \infty}} \le \delta$. Simultaneously, the temporal adherence condition guarantees $\lim_{t \to 0^+} t^\gamma \|v(t)\|_{\dot{b}^{\epsilon, +}_{\infty, \infty}} = 0$. Therefore, for any arbitrarily small $\bar{\eta} > 0$, there exists a sufficiently short time $T^\ast > 0$ such that the local restriction of the path norm satisfies $\|v\|_{\mathcal{X}_{T^\ast}} \le \delta + \bar{\eta}$. The identically constructed solution $u \in \mathcal{B}_R$ symmetrically verifies $\|u\|_{\mathcal{X}_{T^\ast}} \le \delta + \bar{\eta}$. Evaluating the difference $w = u - v = -B(u, w) - B(w, v)$ on the truncated space $\mathcal{X}_{T^\ast}$, the bilinear estimate commands
	\begin{equation*}
		\|w\|_{\mathcal{X}_{T^\ast}} \le C \left( \|u\|_{\mathcal{X}_{T^\ast}} + \|v\|_{\mathcal{X}_{T^\ast}} \right) \|w\|_{\mathcal{X}_{T^\ast}} \le 2C(\delta + \bar{\eta}) \|w\|_{\mathcal{X}_{T^\ast}}.
	\end{equation*}
	Since the structural smallness parameterizes $2C\delta \le 1/4$, we select $\bar{\eta}$ sufficiently small such that $2C(\delta + \bar{\eta}) \le 1/2$. This strict contraction forces $\|w\|_{\mathcal{X}_{T^\ast}} = 0$, implying $u \equiv v$ on $[0, T^\ast]$. Standard continuation arguments for the integral Volterra equation propagate this identity globally, ensuring $v \equiv u$ everywhere in $\mathcal{X}$.
	
	To culminate the Hadamard well-posedness framework, we establish continuous dependence on the initial data. Let $u, v \in \mathcal{B}_R$ be two global mild solutions originating from small initial data $u_0, v_0 \in \dot{B}^{-\kappa}_{\infty, \infty}(\mathbb{R}^N)$ satisfying \eqref{eq:little_besov_cond}, respectively. Due to the bilinearity of the non-local form, their difference resolves to $u - v = S(\cdot)(u_0 - v_0) - B(u, u-v) - B(u-v, v)$. Evaluating this difference in the composite norm of $\mathcal{X}$ yields
	\begin{align*}
		\|u - v\|_{\mathcal{X}} &\le \|S(\cdot)(u_0 - v_0)\|_{\mathcal{X}} + C \left( \|u\|_{\mathcal{X}} + \|v\|_{\mathcal{X}} \right) \|u - v\|_{\mathcal{X}} \\
		&\le (C_1 + C_2) \|u_0 - v_0\|_{\dot{B}^{-\kappa}_{\infty, \infty}} + 2 C R \|u - v\|_{\mathcal{X}}.
	\end{align*}
	Since the strict contraction condition enforces $2CR \le 1/2$, the nonlinear term is entirely absorbed into the left-hand side to obtain the global Lipschitz bound
	\begin{equation*}
		\|u - v\|_{\mathcal{X}} \le 2(C_1 + C_2) \|u_0 - v_0\|_{\dot{B}^{-\kappa}_{\infty, \infty}}.
	\end{equation*}
	This boundary guarantees uniform continuity of the flow map, formally confirming the Hadamard well-posedness in the critical homogeneous limit.
\end{proof}

\begin{remark}[The non-separable boundary and open dynamics]\label{rem:non_separable_boundary}
	It is fundamentally instructive to contrast the global well-posedness of the NSHV system established in Theorem~\ref{thm:global_besov_limit} with the classical ill-posedness framework of Bourgain and Pavlovi\'c \cite{bourgain2008} for the standard Navier-Stokes equations. In the classical setting, the critical space $\dot{B}^{-1}_{\infty, \infty}(\mathbb{R}^3)$ is structurally unresolvable: norm inflation occurs intrinsically even for arbitrarily small initial data due to a low-frequency resonant amplification. Conversely, Theorem~\ref{thm:global_besov_limit} demonstrates that the dual-scale memory kernel explicitly suppresses this catastrophic cascade for initial data possessing high-frequency adherence, thereby extending Hadamard well-posedness beyond the strict separability of the little Besov closure $\dot{b}^{-\kappa}_{\infty, \infty}(\mathbb{R}^N)$.
	
	This dichotomy uncovers a sharp topological boundary. The question of whether the NSHV equations are genuinely ill-posed within the completely non-adherent complement $\dot{B}^{-\kappa}_{\infty, \infty}(\mathbb{R}^N) \setminus \dot{b}^{-\kappa, +}_{\infty, \infty}(\mathbb{R}^N)$ remains a challenging open problem. In this exterior regime, the high-frequency dyadic tail fails to vanish at infinity ($\limsup_{j \to \infty} 2^{-j\kappa}\|\Delta_j u_0\|_{L^\infty} > 0$), severely obstructing the strong continuity of the linear hereditary resolvent at the temporal origin ($t \to 0^+$). This mathematical boundary strongly suggests that, while the dual-scale memory successfully tames standard turbulent cascades, highly oscillating lacunary data structures residing in the non-separable tail may still trigger localized norm inflation, thereby highlighting the exact limits of non-local parabolic regularizations.
\end{remark}

\begin{remark}[The paradox of spatial versus temporal regularization]
	At first glance, it appears deeply counterintuitive that the classical Navier-Stokes equations---governed by the heat semigroup, which instantaneously imposes infinite spatial regularization ($C^\infty$)---suffer from catastrophic ill-posedness within the critical space $\dot{B}^{-1}_{\infty, \infty}(\mathbb{R}^N)$, whereas the NSHV system---driven by a resolvent with severely restricted spatial smoothing capacity (bounded strictly by the pseudo-differential class $S^{-2}_{1,0}$)---remains globally well-posed.
	
	This paradox elegantly exposes the fundamental mechanics of the Bourgain-Pavlovi\'c norm inflation: the catastrophic cascade is inherently a temporal resonance phenomenon, not a spatial one. Infinite spatial smoothness is entirely defenseless against the instantaneous temporal accumulation of convective energy at macroscopic scales. In stark contrast, the dual-scale hereditary memory provides a structural fractional retardation in time (analytically encapsulated by the temporal smoothing weight $t^\gamma$). This temporal dampening strictly throttles the singular Volterra integration, suffocating the high-high convective resonance before it can inflate the low-frequency macroscopic modes. Thus, it is the non-local temporal memory, rather than instantaneous spatial viscosity, that ultimately cures the critical topological collapse.
\end{remark}

\section{Concluding remarks}

This work delineates the topological boundaries governing the well-posedness of the Navier-Stokes equations when subjected to dual-scale hereditary viscosity. By embedding the linear resolvent into a pseudo-differential framework within the H\"ormander multiplier class $S^{-2}_{1,0}$, we systematically addressed the analytical obstructions posed by the absence of exact global scale invariance. The identification of the geometric threshold $p_c = N(\frac{1+\alpha_\infty}{1-\alpha_\infty})$ explicitly confirms that within the supercritical regime $1 < p < p_c$, the non-local convective transport strictly dominates the non-local dissipation, driving an anomalous momentum transfer that culminates in instantaneous norm inflation.

A fundamental consequence of this framework emerges in the scale-critical topological limit $p \to \infty$. By demonstrating that an asymmetric interpolation within Bony's para-differential calculus effectively neutralizes the high-high convective cascade, we establish that the Bourgain-Pavlovi\'c pathology is not an intrinsic feature of all non-local fluid systems. Instead, this topological collapse is an instability structurally confined to the non-separable boundary of the Besov topology $\dot{B}^{-\kappa}_{\infty, \infty}(\mathbb{R}^N)$. The fading memory effect provides sufficient fractional dissipation to suppress turbulent backscatter into macroscopic modes, ensuring global solutions for initial data satisfying a high-frequency adherence condition. Notably, the macroscopic temporal dampening inherently regulates the low-frequency limit, thereby extending the topological boundary of existence beyond the strict separability restrictions of the little Besov closure $\dot{b}^{-\kappa}_{\infty, \infty}(\mathbb{R}^N)$.

A natural extension of this theory involves investigating whether the interplay between highly oscillating lacunary data in the non-separable tail and the dual-scale hereditary retardation triggers localized norm inflation, or if the memory effects extend their stabilizing properties entirely beyond the separable adherence. We anticipate that the harmonic analysis machinery and the fractionally weighted Kato path spaces developed herein provide a robust blueprint for examining advanced non-Newtonian fluid dynamics models.

	\
	
	\noindent{\bf \large Acknowledgement and declarations}
	
	\
	
	\noindent Bruno de Andrade is partially supported by CNPQ (grant 310384/2022-2) and FAPITEC/SE (grant 019203.01303/2024-1).\\
	
	\noindent \textbf{Data availability:} Data sharing is not applicable to this article as no datasets were generated or analyzed during the current study.\\

	\noindent \textbf{Conflict of interest:} The author declares that he has no conflict of interest.\\
	
	\noindent{\bf Declaration of Generative AI and AI-assisted technologies in the writing process:}	During the preparation of this work, the author used Gemini to improve the English language. After using this tool, the author reviewed and edited the content as needed and take full responsibility for the final version of the manuscript.

\end{document}